\documentclass[12pt,reqno]{amsart}
\numberwithin{equation}{section}
\usepackage[T1]{fontenc}
\usepackage{lmodern}
\usepackage{amsmath,amssymb,mathtools}
\usepackage[margin=1.08in]{geometry}
\usepackage{microtype}
\usepackage{enumitem}
\usepackage{needspace}
\usepackage{float}
\usepackage{tikz-cd}
\usetikzlibrary{arrows}
\usepackage{xcolor}
\usepackage[colorlinks=true,linkcolor=blue!55!black,citecolor=blue!55!black,urlcolor=blue!55!black]{hyperref}
\hypersetup{pdftitle={Wide subcategories over commutative coherent rings},pdfauthor={Tiwei Zhao}}
\newtheorem{theorem}{Theorem}[section]
\newtheorem{proposition}[theorem]{Proposition}
\newtheorem{lemma}[theorem]{Lemma}
\newtheorem{corollary}[theorem]{Corollary}
\newtheorem*{conjecture}{Conjecture}
\theoremstyle{definition}

\newtheorem{example}[theorem]{Example}
\theoremstyle{remark}
\newtheorem{remark}[theorem]{Remark}
\newcommand{\Mod}{\operatorname{Mod}}
\newcommand{\modcat}{\operatorname{mod}}
\newcommand{\Dperf}{D_{\mathrm{perf}}}
\newcommand{\Spec}{\operatorname{Spec}}
\newcommand{\Supp}{\operatorname{Supp}}
\newcommand{\supp}{\operatorname{supp}}
\newcommand{\Ann}{\operatorname{Ann}}
\newcommand{\Ass}{\operatorname{Ass}}
\newcommand{\Hom}{\operatorname{Hom}}

\newcommand{\coker}{\operatorname{coker}}
\newcommand{\im}{\operatorname{im}}
\newcommand{\wide}{\operatorname{wide}}
\newcommand{\thick}{\operatorname{thick}}
\newcommand{\Wide}{\operatorname{Wide}}
\newcommand{\Serre}{\operatorname{Serre}}
\newcommand{\Tors}{\operatorname{Tors}}
\newcommand{\tors}{\operatorname{tors}}
\newcommand{\gen}{\operatorname{gen}}
\newcommand{\Thick}{\operatorname{Thick}}
\newcommand{\Thom}{\operatorname{Thom}}
\newcommand{\Torsft}{\operatorname{Tors}_{\mathrm{ft}}}
\newcommand{\ab}{\operatorname{ab}}
\newcommand{\Filt}{\operatorname{Filt}}
\newcommand{\add}{\operatorname{add}}
\newcommand{\W}{\mathcal W}
\newcommand{\C}{\mathcal C}
\newcommand{\T}{\mathcal T}
\newcommand{\Z}{\mathbb Z}
\newcommand{\K}{\mathrm K}
\title[Wide subcategories over commutative coherent rings]{Wide subcategories over commutative coherent rings}

\author[T. Zhao]{Tiwei Zhao}
\address{School of Artificial Intelligence, Jianghan University,
Wuhan 430056, China}
\email{tiweizhao@jhun.edu.cn}

\date{}
\begin{document}
\begin{abstract}
In this paper, we mainly prove Hovey's conjecture: for any commutative coherent ring, there is a lattice isomorphism between the lattice of wide subcategories of the category of finitely presented modules and that of thick subcategories of the perfect derived category.  We  construct, for every finitely presented module, a bounded finite free complex whose zeroth homology is the given module and whose homology lies in its abelian closure. This realization gives an inverse to Hovey's correspondence without a closure operation.  In addition, we study finite extension filtrations, their behavior under change of rings, and the limits of reconstructing wide subcategories of arbitrary modules from finitely presented modules.
\end{abstract}
\maketitle

\section{Introduction}


All rings are commutative with identity unless stated otherwise, and modules are unital. An $R$-module is \emph{finitely presented} if it admits an exact sequence $R^a\to R^b\to M\to0$ with $a,b$ finite. A ring is \emph{coherent} if each of its finitely generated ideals is finitely presented as a module \cite[Definition~4.45]{Lam}. Let $R$ be such a ring. Write $\Mod R$ for the category of $R$-modules, $\modcat R$ for its abelian subcategory of finitely presented modules, and $\operatorname{proj}R$ for the subcategory of finitely generated projective modules. A \emph{perfect complex} is a complex quasi-isomorphic to a bounded complex in $\operatorname{proj}R$; following \cite[Section~2]{Takahashi}, we write $\Dperf(R)$ for the full subcategory of the derived category of $R$-modules consisting of perfect complexes, and identify it with the bounded homotopy category $\K^b(\operatorname{proj}R)$. Complexes are homologically indexed: $H_i$ denotes homology in degree $i$, and $X[n]_i=X_{i-n}$ with differential $(-1)^n d_X$. We regard a module as a complex in degree zero and write $\otimes_R^{\mathbf L}$ for the derived tensor product. All subcategories are nonempty, full and closed under isomorphisms; classes of objects include all isomorphic objects.

In an abelian category, an \emph{abelian subcategory} is closed under finite direct sums, kernels and cokernels, so its inclusion is exact. It is \emph{wide} if it is also closed under extensions, and \emph{Serre} if it is closed under subobjects, quotients and extensions; compare \cite[Definition~1.1]{Hovey} and \cite[Definition~2.3(1),(2)]{Takahashi}. A \emph{torsion class} in $\modcat R$ is closed under quotients and extensions \cite[Definition~3.1(2)]{IK}; this does not presuppose a torsion pair. Takahashi uses this term for hereditary torsion classes in $\Mod R$ \cite[Definition~2.3(3)]{Takahashi}, which we define below. Unless another ambient category is specified, these closure conditions concern objects of $\modcat R$. Write $\Wide(\modcat R)$, $\Serre(\modcat R)$ and $\Tors(\modcat R)$ for the respective lattices, ordered by inclusion. A \emph{thick subcategory} in $\Dperf(R)$ is a triangulated subcategory closed under direct summands \cite[Definition~2.10(1)]{Takahashi}; their lattice in $\Dperf(R)$ is denoted by $\Thick(\Dperf(R))$.

We write $\Spec R$ for the prime spectrum of $R$. For a module $M$ and a perfect complex $P$, we use ordinary localization supports
\[
 \Supp_R M=\{\mathfrak p\in\Spec R:M_{\mathfrak p}\ne0\},\quad
 \Supp_R P=\{\mathfrak p\in\Spec R:P_{\mathfrak p}\not\simeq0\}.
\]
For a module or complex subcategory $\mathcal E$, we set
$\Supp_R(\mathcal E)=\bigcup_{E\in\mathcal E}\Supp_R E$.
For an ideal $I$, we put $V(I)=\{\mathfrak p\in\Spec R:I\subseteq\mathfrak p\}$. A subset $V\subseteq\Spec R$ is \emph{specialization-closed} if $\mathfrak p\in V$ and $\mathfrak p\subseteq\mathfrak q$ imply $\mathfrak q\in V$. A \emph{Thomason subset} is a union of sets $V(I)$ with $I$ finitely generated, equivalently a union of closed subsets with quasi-compact open complements \cite[Section~2]{GPtorsion}. Their complete lattice under inclusion is denoted by $\Thom(\Spec R)$.

The classification of thick subcategories of perfect complexes by support originates in the work of Hopkins \cite{Hopkins}. For commutative noetherian rings, the Hopkins--Neeman theorem identifies these subcategories with specialization-closed subsets of the prime spectrum \cite[Theorem~1.5]{Neeman}. Thomason extended this classification to arbitrary commutative rings, with Thomason subsets replacing specialization-closed subsets \cite[Theorem~3.15]{Thomason}. Motivated by this classification, Hovey \cite{Hovey} introduced the homology map
\begin{equation}\label{eq:f}
 \begin{aligned}
 f_R:\Wide(\modcat R)&\longrightarrow\Thick(\Dperf(R)),\\
 \W&\longmapsto\{P\in\Dperf(R):H_i(P)\in\W\text{ for every }i\}.
 \end{aligned}
\end{equation}

For the classification diagram below, we write $\Torsft(\Mod R)$ for the lattice of hereditary torsion classes of finite type. A \emph{hereditary torsion class} $\mathcal U\subseteq\Mod R$ is closed under submodules, quotients, extensions and arbitrary direct sums. It is \emph{of finite type} if its torsion-free class
\[
 \mathcal U^\perp=\{Y\in\Mod R:\Hom_R(X,Y)=0\text{ for all }X\in\mathcal U\}
\]
is closed under filtered colimits; see \cite[Appendix]{GPtorsion}. The following Figure~\ref{fig:classification-overview} collects the classification correspondences used in this paper and the known noetherian closure results. 
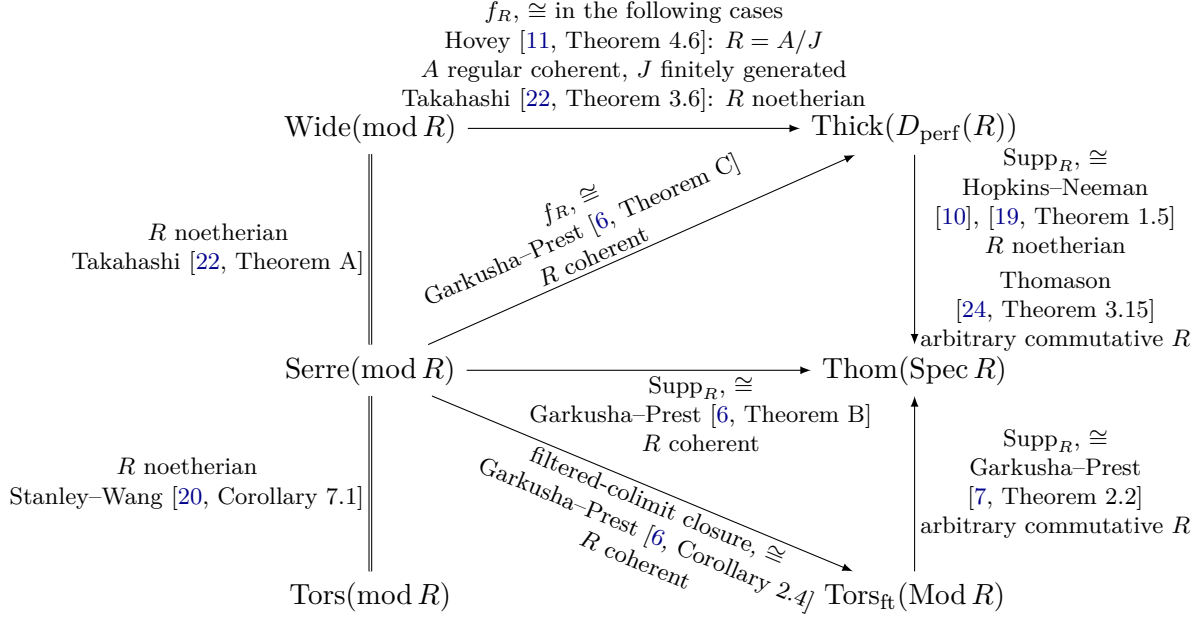
\begin{figure}[htbp]
\begin{tikzpicture}[
 >=latex,
 category/.style={inner sep=4pt,font=\small},
 arrowlabel/.style={font=\scriptsize,align=center,inner sep=2pt}
]
 \node[category] (wide) at (0,0) {$\Wide(\modcat R)$};
 \node[category] (perfect) at (7.2,0) {$\Thick(\Dperf(R))$};
 \node[category] (serre) at (0,-3.2) {$\Serre(\modcat R)$};
 \node[category] (support) at (7.2,-3.2) {$\Thom(\Spec R)$};
 \node[category] (torsion) at (0,-6.2) {$\Tors(\modcat R)$};
 \node[category] (large) at (7.2,-6.2) {$\Torsft(\Mod R)$};

 \draw[->] (wide) -- node[arrowlabel,above,yshift=4pt]
   {$f_R$, $\cong$ in the following cases\\
    Hovey \cite[Theorem~4.6]{Hovey}: $R=A/J$\\
    $A$ regular coherent, $J$ finitely generated\\
    Takahashi \cite[Theorem~3.6]{Takahashi}: $R$ noetherian} (perfect);
 \draw[double] (serre) -- node[arrowlabel,left]
   {$R$ noetherian\\Takahashi \cite[Theorem~A]{Takahashi}} (wide);
 \draw[double] (serre) -- node[arrowlabel,left]
   {$R$ noetherian\\Stanley--Wang \cite[Corollary~7.1]{SW}} (torsion);
 \draw[->] (serre) -- node[arrowlabel,below,pos=.68]
   {$\Supp_R$, $\cong$\\Garkusha--Prest \cite[Theorem~B]{GP}\\$R$ coherent} (support);
 \draw[->] (serre) -- node[arrowlabel,sloped,above,pos=.4]
   {$f_R$, $\cong$\\Garkusha--Prest \cite[Theorem~C]{GP}\\$R$ coherent} (perfect);
 \draw[->] (serre) -- node[arrowlabel,sloped,below,pos=.55]
   {filtered-colimit closure, $\cong$\\Garkusha--Prest \cite[Corollary~2.4]{GP}\\$R$ coherent} (large);
 \draw[->] (perfect) -- node[arrowlabel,right]
   {$\Supp_R$, $\cong$\\
    Hopkins--Neeman\\
    \cite{Hopkins}, \cite[Theorem~1.5]{Neeman}\\
    $R$ noetherian\\[2pt]
    Thomason\\\cite[Theorem~3.15]{Thomason}\\arbitrary commutative $R$} (support);
 \draw[->] (large) -- node[arrowlabel,right]
   {$\Supp_R$, $\cong$\\Garkusha--Prest\\\cite[Theorem~2.2]{GPtorsion}\\arbitrary commutative $R$} (support);
\end{tikzpicture}
\caption{Classification correspondences for commutative rings. Arrows marked $\cong$ are complete lattice isomorphisms. }
\label{fig:classification-overview}
\end{figure}

There are other related works, such as noetherian closure criteria \cite{IMST}, intervals in torsion lattices of length categories \cite{AP}, and radical wide subcategories relative to a chosen family \cite{TWZ}. In Figure~\ref{fig:classification-overview}, a coherent ring is \emph{regular} if every finitely presented module has finite projective dimension \cite[Section~3]{Hovey}.
 In the discussion after \cite[Corollary~4.7]{Hovey}, Hovey proposed extending bijectivity to every commutative coherent ring. Takahashi records the following conjecture in the introduction of \cite{Takahashi}, where the term \emph{coherent subcategory} means a wide subcategory here.

\begin{conjecture}[Hovey; {\cite[Introduction]{Takahashi}}]
For a commutative coherent ring $R$, the lattices below are isomorphic:
\[
 \left\{\begin{array}{c}
 \text{thick subcategories}\\
 \text{of }\Dperf(R)
 \end{array}\right\}
 \ \cong\ 
 \left\{\begin{array}{c}
 \text{coherent(=wide) subcategories}\\
 \text{of }\modcat R
 \end{array}\right\}.
\]
\end{conjecture}

\Needspace{20\baselineskip}
In this paper, we will prove that every wide subcategory of $\modcat R$ is Serre over every commutative coherent ring $R$. Combined with the Serre/perfect-complex correspondence of Garkusha and Prest \cite[Theorem~C]{GP}, this proves the conjecture. We also give a construction that identifies the inverse of $f_R$ without taking a closure. The main results are as follows:

\begin{theorem}\label{thm:main}
Let $R$ be a commutative coherent ring.
\begin{enumerate}[label=\textup{(\arabic*)},leftmargin=*]
\item There are equalities and a natural isomorphism of complete lattices
\[
 \Wide(\modcat R)=\Serre(\modcat R)=\Tors(\modcat R)
 \cong\Torsft(\Mod R),
\]
where the last isomorphism sends $\W$ to the class of all filtered colimits of its objects in $\Mod R$, and its inverse sends $\mathcal U$ to $\mathcal U\cap\modcat R$.
\item The map $f_R$ is an isomorphism of complete lattices, with inverse
\begin{equation}\label{eq:inverse}
 g_R(\T)=\{H_0(P):P\in\T\}.
\end{equation}
\item The following diagram of complete lattice isomorphisms commutes:
\begin{equation}\label{eq:main-lattices}
 \begin{tikzcd}[column sep=small,row sep=small]
 \Wide(\modcat R)
   \arrow[rr,shift left=.7ex,"f_R"]
   \arrow[dr,"\Supp_R"']
 && \Thick(\Dperf(R))
   \arrow[ll,shift left=.7ex,"g_R"]
   \arrow[dl,"\Supp_R"]
 \\
 & \Thom(\Spec R) &
 \end{tikzcd}
\end{equation}
\end{enumerate}
\end{theorem}

Once part~(1) is known, the lattice isomorphism in part~(2) and the support diagram in part~(3) follow from Garkusha and Prest \cite[Theorems~B and~C]{GP}. Their inverse sends $\T$ to the Serre subcategory generated by $\{H_i(P):P\in\T,\ i\in\Z\}$. Formula~\eqref{eq:inverse} strengthens this description by removing the closure operation.

In Section~\ref{sec:finite}, we prove the closure results by finite cyclic descent and reduction to a noetherian subring, extending the results of Takahashi \cite[Theorem~3.1]{Takahashi} and Stanley--Wang \cite[Corollary~7.1]{SW}; we also give the ideal-filter classification and annihilator-power formulas. In Section~\ref{sec:perfect}, we construct perfect complexes with prescribed zeroth homology and all homology in the module's abelian closure. Combined with Hovey's surjectivity theorem, this yields the inverse~\eqref{eq:inverse} and completes the proof of Theorem~\ref{thm:main}. In Section~\ref{sec:basechange}, we study change of rings and compare quotient-and-extension filtration lengths through iterated traces. In Section~\ref{sec:large}, we  apply the Garkusha--Prest correspondence \cite[Corollary~2.4]{GP} to identify the largest finite-type hereditary torsion class contained in a wide subcategory closed under arbitrary sums, and describe restriction to finitely presented modules in Krause's classification \cite[Theorem~3.1]{Krause}. Appendix~\ref{app:associated-primes} treats the noetherian boundary for parametrization by associated primes. 

We also need the following notations in this paper. For a family $\mathcal S\subseteq\modcat R$, write $\ab_R(\mathcal S)$, $\wide_R(\mathcal S)$ and $\tors_R(\mathcal S)$ for the smallest abelian, wide and torsion subcategories containing $\mathcal S$. For $\mathcal E\subseteq\Dperf(R)$, write $\thick_R(\mathcal E)$ for its thick closure. We omit braces for singleton families. For any module $M$, put $\Ann_R M=\{r\in R:rM=0\}$ and $\Ann_R(x)=\{r\in R:rx=0\}$ for $x\in M$. Write $\kappa(\mathfrak p)=R_{\mathfrak p}/\mathfrak pR_{\mathfrak p}$ for the residue field at $\mathfrak p$, and $\Ass_R M$ for the ordinary associated primes, namely the prime ideals of the form $\Ann_R(x)$ with $0\ne x\in M$ \cite[Definition~2.1]{ES}. 

\section{Abelian subcategories and torsion classes}\label{sec:finite}

The finite descent argument underlying the classification works over an arbitrary commutative ring.

\begin{proposition}\label{prop:finite-submodules}
Let $R$ be commutative and let $\C\subseteq\Mod R$ be a class of finitely presented modules. Assume that $\C$ is closed under finite direct sums, kernels and cokernels. If $M\in\C$ and $N\subseteq M$ is finitely generated, then $N,M/N\in\C$. In particular, $N$ is finitely presented.
\end{proposition}
\begin{proof}
First let $B\in\C$ and consider a finitely generated submodule $A$ of $B$ with cyclic quotient $B/A\cong R/I$. By \cite[Lemma~4.54(1)]{Lam}, $R/I$ is finitely presented, and by \cite[Proposition~4.26(b)]{Lam}, $I$ is finitely generated. Write $I=(a_1,\ldots,a_r)$. Since $R$ is commutative, the scalar map $(a_1,\ldots,a_r): B^r\to B$ is $R$-linear, and closure under cokernels gives
\begin{equation}\label{eq:scalar-cokernel}
 B/IB=\coker\left(B^r\xrightarrow{(a_1,\ldots,a_r)}B\right)\in\C.
\end{equation}
Since $IB\subseteq A$, the quotient $q:B\twoheadrightarrow R/I$ factors through an $R/I$-linear surjection $\overline q:B/IB\twoheadrightarrow R/I$. This surjection splits: if $q(b)=1+I$, then
$s(x+I)=xb+IB$ defines a section. Consequently $R/I$ is a direct summand of $B/IB$. Since closure under kernels implies closure under direct summands, $R/I$ belongs to $\C$. Closure under kernels then gives $A=\ker q\in\C$.

Now let $M\in\C$ and $N\subseteq M$ be finitely generated. Lift a finite generating set $x_1,\ldots,x_t$ of $M/N$ and put
\begin{align*}
N_0&=N\\
N_1&=N+Rx_1\\
&\cdots \\
N_t&=N+Rx_1+Rx_2+\cdots+Rx_t=M.
\end{align*}
Then each $N_i$ is finitely generated and $N_i/N_{i-1}$ is cyclic. Descending from $N_t=M\in\C$, the preceding argument gives $N_{i-1}\in\C$ at every step. Finally, we obtain $N\in\C$, and the cokernel of $N\hookrightarrow M$ gives $M/N\in\C$.
\end{proof}

From now on $R$ is a commutative coherent ring. Then $\modcat R$ is an abelian category \cite[Lemma~1.6 and its proof]{Hovey}; see also \cite[Sections~2.2--2.3]{Glaz}. Applying Proposition~\ref{prop:finite-submodules} to $\C=\modcat R$ shows that finitely generated submodules of finitely presented modules are finitely presented. In particular, if $M\in\modcat R$ is generated by $m_1,\ldots,m_s$, its annihilator is finitely generated, since
\begin{equation}\label{eq:ann-kernel}
 \Ann_R M=\ker\bigl(R\longrightarrow M^s,\ a\longmapsto(am_1,\ldots,am_s)\bigr).
\end{equation}
Proposition~\ref{prop:finite-submodules} also shows that every abelian subcategory of $\modcat R$ is closed under subobjects and quotients, and hence every wide subcategory is Serre. We now classify these abelian subcategories and describe their extension closures.

We record the elementary restriction-of-scalars fact used in the annihilator arguments; compare \cite[Lemma~4.54(1)]{Lam}.

\Needspace{8\baselineskip}
\begin{lemma}\label{lem:quotient-modules}
Let $I$ be a finitely generated ideal of $R$. Restriction of scalars identifies $\modcat(R/I)$ with the abelian subcategory
\[
 \mathcal A_I=\{N\in\modcat R:IN=0\}.
\]
If a full subcategory $\C\subseteq\modcat R$ contains $R/I$ and is closed under finite direct sums and cokernels, then $\mathcal A_I\subseteq\C$.
\end{lemma}
\begin{proof}
Let $N$ be finitely presented over $R$ and satisfy $IN=0$. Tensoring a finite presentation of $N$ with $R/I$ gives
\[
 (R/I)^a\longrightarrow(R/I)^b\longrightarrow N\longrightarrow0,
\]
so $N$ is finitely presented over $R/I$. Conversely, a finite presentation over $R/I$ expresses $N$ as the quotient of $(R/I)^b$ by a finitely generated $R$-submodule. Since $I$ is finitely generated, $(R/I)^b$ is finitely presented over $R$; hence so is $N$ by \cite[Lemma~4.54(1)]{Lam}.

Every $R$-linear map between modules killed by $I$ is $R/I$-linear, so restriction of scalars is fully faithful with the stated image. Finite direct sums, kernels and cokernels of such modules are still killed by $I$ and are finitely presented by coherence. Thus $\mathcal A_I$ is abelian. Finally, if $R/I\in\C$, the displayed presentation and closure under finite direct sums and cokernels give $N\in\C$ for every $N\in\mathcal A_I$.
\end{proof}

Let $\mathcal I_{\mathrm{fg}}(R)$ be the partially ordered set of finitely generated ideals. It is closed under finite sums, intersections and products. A \emph{filter} in $\mathcal I_{\mathrm{fg}}(R)$ is a family $\mathcal F$ containing $R$, upward closed among finitely generated ideals, and closed under finite intersections. Here upward closure among finitely generated ideals means that if $I\in\mathcal F$ and $J$ is a finitely generated ideal with $I\subseteq J$, then $J\in\mathcal F$. A filter is \emph{multiplicative} if it is closed under ideal products.

The following theorem is the finitely presented form of the classical correspondence between ideal filters and subcategories closed under submodules, quotients and sums; see \cite[Corollary~9.4]{Kanda} for the correspondence in $\Mod R$. Proposition~\ref{prop:finite-submodules} makes it applicable here to every abelian subcategory, and the finite-generation assumptions give the stated elementary criterion for extension closure.

\begin{theorem}\label{thm:filters}
The assignments
\begin{align*}
 \C&\longmapsto\mathcal F_{\C}
     =\{I\in\mathcal I_{\mathrm{fg}}(R):R/I\in\C\},\\
 \mathcal F&\longmapsto\C_{\mathcal F}
     =\{M\in\modcat R:\Ann_R M\in\mathcal F\}
\end{align*}
are mutually inverse order-preserving bijections between abelian subcategories of $\modcat R$ and filters in $\mathcal I_{\mathrm{fg}}(R)$. Under these bijections, the following conditions are equivalent:
\begin{enumerate}[label=\textup{(\arabic*)},leftmargin=*]
\item $\C$ is wide;
\item $\C$ is Serre;
\item $I,J\in\mathcal F_{\C}$ implies $IJ\in\mathcal F_{\C}$, that is, $\mathcal F_{\C}$ is multiplicative;
\item $I\in\mathcal F_{\C}$ implies $I^2\in\mathcal F_{\C}$.
\end{enumerate}
Thus extension closure is equivalent to the implication
\[
 R/I\in\C\ \Longrightarrow\ R/I^2\in\C
 \quad\text{for every finitely generated ideal }I.
\]
\end{theorem}
\begin{proof}
Let $\C$ be an abelian subcategory of $\modcat R$. Since $R/R=0\in\C$, we have $R\in\mathcal F_{\C}$. If $I\in\mathcal F_{\C}$ and $I\subseteq J$ with $J$ finitely generated, the natural surjection $R/I\twoheadrightarrow R/J$ and Proposition~\ref{prop:finite-submodules} give $R/J\in\C$. Thus $\mathcal F_{\C}$ is upward closed among finitely generated ideals. If $I,J\in\mathcal F_{\C}$, then
\[
 R/(I\cap J)\hookrightarrow R/I\oplus R/J.
\]
The target belongs to $\C$, so Proposition~\ref{prop:finite-submodules} gives $R/(I\cap J)\in\C$. Hence $\mathcal F_{\C}$ is a filter.

Conversely, let $\mathcal F$ be a filter, let $M\in\C_{\mathcal F}$, and let $L\subseteq M$ be a subobject in $\modcat R$. Then
\[
 \Ann_R M\subseteq\Ann_R L,
 \quad \Ann_R M\subseteq\Ann_R(M/L).
\]
All three ideals are finitely generated by~\eqref{eq:ann-kernel}, so upward closure gives $L,M/L\in\C_{\mathcal F}$. Thus $\C_{\mathcal F}$ is closed under subobjects and quotients, and in particular under kernels and cokernels. It contains zero and is closed under finite direct sums because
$\Ann_R(M\oplus N)=\Ann_R M\cap\Ann_R N$. Therefore it is an abelian subcategory.

To show that the assignments are inverse, choose $s$ generators of $M\in\modcat R$. Equation~\eqref{eq:ann-kernel} gives an injection $R/\Ann_R M\hookrightarrow M^s$, and the chosen generators give a surjection $(R/\Ann_R M)^s\twoheadrightarrow M$. Proposition~\ref{prop:finite-submodules} therefore shows that every abelian subcategory $\C$ satisfies
\begin{equation}\label{eq:ann-test}
 M\in\C\quad\Longleftrightarrow\quad R/\Ann_R M\in\C.
\end{equation}
Consequently $\C_{\mathcal F_{\C}}=\C$. Since $\Ann_R(R/I)=I$, we also have $\mathcal F_{\C_{\mathcal F}}=\mathcal F$. Both assignments preserve inclusion.

We now prove the equivalences. Conditions~\textup{(1)} and~\textup{(2)} are equivalent by Proposition~\ref{prop:finite-submodules}.

\textup{(1)$\Rightarrow$(4).} Let $I\in\mathcal F_{\C}$. The module $I/I^2$ is finitely presented and annihilated by $I$, so Lemma~\ref{lem:quotient-modules} gives $I/I^2\in\C$. The exact sequence
\[
 0\longrightarrow I/I^2\longrightarrow R/I^2\longrightarrow R/I\longrightarrow0
\]
and extension closure imply $R/I^2\in\C$, proving~\textup{(4)}.

\textup{(4)$\Rightarrow$(3).} If $I,J\in\mathcal F_{\C}$, put $K=I\cap J$. Then $K\in\mathcal F_{\C}$, so $K^2\in\mathcal F_{\C}$ by~\textup{(4)}. Since $K^2\subseteq IJ$, upward closure gives $IJ\in\mathcal F_{\C}$.

\textup{(3)$\Rightarrow$(1).} Let $0\to L\to M\to N\to0$ be exact with $L,N\in\C$. By~\eqref{eq:ann-test} and~\textup{(3)}, the ideal $(\Ann_R L)(\Ann_R N)$ belongs to $\mathcal F_{\C}$. If $b\in\Ann_R N$ and $m\in M$, then $bm\in L$, so $abm=0$ for every $a\in\Ann_R L$. Hence
\[
 (\Ann_R L)(\Ann_R N)\subseteq\Ann_R M.
\]
Upward closure gives $\Ann_R M\in\mathcal F_{\C}$, and~\eqref{eq:ann-test} yields $M\in\C$. Thus $\C$ is wide.
\end{proof}

The commutativity and finiteness hypotheses are essential for proving that a wide subcategory is Serre. For example, let $k$ be a field. Write $\add E$ for the direct summands of finite direct sums of $E$, and consider the projective representation $E=(k\xrightarrow{1}k)$ of the quiver $1\to2$. Every object of $\add E$ is isomorphic to $(k^r\xrightarrow{1}k^r)$ for some $r\ge0$. A morphism between two such objects is given by the same linear map $u$ at both vertices. Its kernel and cokernel are therefore $(\ker u\xrightarrow{1}\ker u)$ and $(\coker u\xrightarrow{1}\coker u)$, which again belong to $\add E$. Moreover, any extension between objects of $\add E$ splits because the quotient is projective. Thus $\add E$ is wide. However, the subrepresentation $(0\to k)$ of $E$ does not belong to $\add E$, so this subcategory is not Serre. For the failure without finite presentation, see the $\mathbb Q$-vector-space example after Proposition~\ref{prop:krause-trace}.

For a subcategory $\mathcal A$ containing $0$ and $n\ge1$, let $\Filt_n(\mathcal A)$ consist of the modules $N$ admitting a filtration $0=N_0\subseteq N_1\subseteq\cdots\subseteq N_m=N$ with $m\le n$ and $N_i/N_{i-1}\in\mathcal A$, and set $\Filt(\mathcal A)=\bigcup_{n\ge1}\Filt_n(\mathcal A)$. All filtrations below take place in $\modcat R$; compare \cite[Section~2]{MS} for the unbounded operation $\Filt$.

\begin{proposition}\label{prop:generation}
Let $M\in\modcat R$. Then, for every $n\ge1$,
\begin{align}
 \ab_R(M)&=\ab_R(R/\Ann_R M)
 =\{N\in\modcat R:(\Ann_R M)N=0\},\label{eq:ab-generator}\\
 \Filt_n(\ab_R(M))&=\{N\in\modcat R:(\Ann_R M)^nN=0\},\label{eq:filtration}\\
 \wide_R(M)&=\Filt(\ab_R(M))
 =\{N\in\modcat R:\Supp_R N\subseteq\Supp_R M\}.\label{eq:wide-generator}
\end{align}
More generally, let $\mathcal S$ be any family of finitely presented modules and, for a finite subfamily $F\subseteq\mathcal S$, put $I_F=\bigcap_{M\in F}\Ann_R M$, with $I_{\varnothing}=R$. Then
\begin{align}
 \ab_R(\mathcal S)&=\{N\in\modcat R:I_FN=0\text{ for some finite }F\subseteq\mathcal S\},\label{eq:ab-family}\\
 \wide_R(\mathcal S)&=\Filt(\ab_R(\mathcal S))\notag\\
 &=\{N\in\modcat R:I_F^nN=0\text{ for some finite }F\subseteq\mathcal S,\ n\ge1\}.\label{eq:wide-family}
\end{align}
\end{proposition}
\begin{proof}
By~\eqref{eq:ann-test}, $R/\Ann_R M$ belongs to $\ab_R(M)$. Lemma~\ref{lem:quotient-modules} identifies the finitely presented modules annihilated by $\Ann_R M$ with an abelian subcategory containing $M$. It also shows that every module in this subcategory is obtained from $R/\Ann_R M$ using finite direct sums and cokernels. Hence
\[
 \ab_R(M)=\ab_R(R/\Ann_R M)
 =\{N\in\modcat R:(\Ann_R M)N=0\},
\]
proving~\eqref{eq:ab-generator}.

We next prove~\eqref{eq:filtration}, beginning with $n=2$. If $N\in\Filt_2(\ab_R(M))$, choose a filtration $0\subseteq N_1\subseteq N$ with $N_1,N/N_1\in\ab_R(M)$. By~\eqref{eq:ab-generator},
\[
 (\Ann_R M)N_1=0,\quad (\Ann_R M)(N/N_1)=0.
\]
The second equality gives $(\Ann_R M)N\subseteq N_1$, and hence
\[
 (\Ann_R M)^2N\subseteq(\Ann_R M)N_1=0.
\]
Similarly, for a filtration $0=N_0\subseteq N_1\subseteq\cdots\subseteq N_t=N$ with $t\le n$ and factors in $\ab_R(M)$, we have $(\Ann_R M)N_j\subseteq N_{j-1}$ at each step. Repeating the preceding argument gives $(\Ann_R M)^tN=0$, and therefore $(\Ann_R M)^nN=0$.

Conversely, let $N\in\modcat R$ satisfy $(\Ann_R M)^nN=0$. Put $N_j=(\Ann_R M)^{n-j}N$ for $0\le j\le n$, so that
\[
 0=N_0\subseteq N_1\subseteq\cdots\subseteq N_n=N.
\]
Each $N_j$ is finitely generated because $\Ann_R M$ and $N$ are finitely generated; coherence then shows that the terms and successive quotients are finitely presented. Starting at the left end, we have $(\Ann_R M)N_1=(\Ann_R M)^nN=0$, so $N_1\in\ab_R(M)$ by~\eqref{eq:ab-generator}. This proves the assertion when $n=1$. If $n\ge2$, then $(\Ann_R M)N_2=N_1$, so $N_2/N_1$ is also annihilated by $\Ann_R M$ and belongs to $\ab_R(M)$. Thus $0\subseteq N_1\subseteq N_2$ is a filtration of $N_2$ with two factors in $\ab_R(M)$. Repeating this step, using $(\Ann_R M)N_j=N_{j-1}$, shows that every quotient $N_j/N_{j-1}$ belongs to $\ab_R(M)$. Hence $N=N_n\in\Filt_n(\ab_R(M))$, completing the proof of~\eqref{eq:filtration}.

To prove~\eqref{eq:wide-generator}, consider the class
\[
 \mathcal D=\{N\in\modcat R:(\Ann_R M)^nN=0\text{ for some }n\ge1\}.
\]
The finitely generated ideals containing a power of $\Ann_R M$ form a multiplicative filter. Theorem~\ref{thm:filters} therefore shows that $\mathcal D$ is Serre. By~\eqref{eq:filtration}, it equals $\Filt(\ab_R(M))$. Since this class contains $M$ and is contained in every wide subcategory containing $M$, it equals $\wide_R(M)$, proving the first equality in~\eqref{eq:wide-generator}.

For the support description, recall that $\Supp_R L=V(\Ann_R L)$ for every finitely generated module $L$. Therefore
\[
 \begin{aligned}
  \Supp_R N\subseteq\Supp_R M
  &\quad\Longleftrightarrow\quad V(\Ann_R N)\subseteq V(\Ann_R M)\\
  &\quad\Longleftrightarrow\quad \Ann_R M\subseteq\sqrt{\Ann_R N}.
 \end{aligned}
\]
If $(\Ann_R M)^nN=0$, then $(\Ann_R M)^n\subseteq\Ann_R N$, which gives the last inclusion. This shows that $\Filt(\ab_R(M))\subseteq\{N\in\modcat R:\Supp_R N\subseteq\Supp_R M\}$. Conversely, suppose $\Ann_R M\subseteq\sqrt{\Ann_R N}$. If $\Ann_R M=0$, take $n=1$. Otherwise, choose generators $a_1,\ldots,a_r$ of $\Ann_R M$ and integers $e_i\ge1$ with $a_i^{e_i}\in\Ann_R N$. For $n=1+\sum_{i=1}^r(e_i-1)$, every monomial of degree $n$ in the generators contains some $a_i^{e_i}$ as a factor. Consequently $(\Ann_R M)^n\subseteq\Ann_R N$, so $(\Ann_R M)^nN=0$. This proves the reverse inclusion and hence the support equality in~\eqref{eq:wide-generator}.

We now turn to an arbitrary family $\mathcal S$. For a finite subfamily $F\subseteq\mathcal S$, put $M_F=\bigoplus_{L\in F}L$, with $M_{\varnothing}=0$. A subcategory closed under finite direct sums and direct summands contains $F$ if and only if it contains $M_F$, so we have 
$\ab_R(F)=\ab_R(M_F)$, $\wide_R(F)=\wide_R(M_F)$ and  $\Ann_R M_F=I_F$.
The subcategories generated by finite subfamilies form a directed family: those indexed by $F$ and $G$ are both contained in the one indexed by $F\cup G$. Thus any two objects in their union belong to a common such subcategory, where finite direct sums, kernels and cokernels can be computed; in the wide case, extensions remain in that same subcategory as well. Therefore, the unions are abelian and wide, respectively, and contain $\mathcal S$. Each term is contained in the corresponding subcategory generated by $\mathcal S$, and the reverse inclusions follow from minimality. Hence
\[
 \ab_R(\mathcal S)=\bigcup_{F\subseteq\mathcal S,\ F\text{ finite}}\ab_R(M_F)
 \quad \text{and }
 \wide_R(\mathcal S)=\bigcup_{F\subseteq\mathcal S,\ F\text{ finite}}\wide_R(M_F).
\]
Applying~\eqref{eq:ab-generator} to each $M_F$ proves~\eqref{eq:ab-family}. Similarly, \eqref{eq:filtration} and~\eqref{eq:wide-generator} show that $N\in\wide_R(\mathcal S)$ if and only if $I_F^nN=0$ for some finite $F\subseteq\mathcal S$ and $n\ge1$, proving the last expression in~\eqref{eq:wide-family}.

Finally, if $I_F^nN=0$, apply~\eqref{eq:filtration} to $M_F$ to obtain
\[
 N\in\Filt_n(\ab_R(M_F))\subseteq\Filt(\ab_R(\mathcal S)).
\]
Conversely, $\wide_R(\mathcal S)$ contains $\ab_R(\mathcal S)$ and is closed under extensions, so it contains every finite filtration by objects of $\ab_R(\mathcal S)$. This proves the remaining equality in~\eqref{eq:wide-family}.
\end{proof}

The preceding formulas and the structure of finitely generated idempotent ideals give the following criterion.

\begin{corollary}\label{cor:idempotent}
For $M\in\modcat R$, the following are equivalent:
\begin{itemize}
 \item[{\rm (i)}] $\ab_R(M)=\wide_R(M)$, 
 \item[{\rm (ii)}] $(\Ann_R M)^2=\Ann_R M$,
 \item[{\rm (iii)}] $\Ann_R M=Re$ for an idempotent $e$.
\end{itemize}
In this case $\ab_R(M)$ is the category of finitely presented modules over the direct factor $R(1-e)$, regarded as $R$-modules. In particular, every faithful finitely presented module generates $\modcat R$ using finite direct sums, kernels and cokernels alone.
\end{corollary}
\begin{proof}
\textup{(i)$\Rightarrow$(ii).} By Proposition~\ref{prop:generation} and (i), we have
\(
 R/(\Ann_R M)^2\in\wide_R(M)=\ab_R(M).
\)
Then by~\eqref{eq:ab-generator}, it is annihilated by $\Ann_R M$, which means that $\Ann_R M\subseteq(\Ann_R M)^2$. The reverse inclusion always holds, proving (ii).

\textup{(ii)$\Rightarrow$(iii).} The ideal $\Ann_R M$ is finitely generated by~\eqref{eq:ann-kernel}. Every finitely generated idempotent ideal of a commutative ring is generated by an idempotent \cite[Lemma~2.43]{Lam}, so~\textup{(ii)} implies~\textup{(iii)}.

\textup{(iii)$\Rightarrow$(i).} Since $e^2=e$, for every $n\ge1$ we have
\[
 (\Ann_R M)^n=(Re)^n=Re=\Ann_R M.
\]
Consequently~\eqref{eq:ab-generator} and~\eqref{eq:filtration} give $\Filt_n(\ab_R(M))=\ab_R(M)$ for every $n\ge1$. Taking the union over $n$ and using~\eqref{eq:wide-generator}, we obtain $\wide_R(M)=\ab_R(M)$, proving (i).

Finally, the ring isomorphism $R/\Ann_R M\xrightarrow{\sim}R(1-e)$, given by $r+\Ann_R M\mapsto r(1-e)$, together with~\eqref{eq:ab-generator} identifies $\ab_R(M)$ with the category of finitely presented modules over $R(1-e)$. If $M$ is faithful, then $\Ann_R M=0$ and~\eqref{eq:ab-generator} gives $\ab_R(M)=\modcat R$.
\end{proof}

We next descend a finite presentation to apply the commutative noetherian torsion-class theorem of Stanley and Wang \cite[Corollary~7.1]{SW}, also proved in \cite[Corollary~3.23(b)]{IK}.

\begin{theorem}\label{thm:torsion-serre}
Every torsion class in $\modcat R$ is Serre. More precisely, if $M$ belongs to a torsion class $\T$, then $R/\Ann_R M$ belongs to $\T$, and so does every finitely presented module annihilated by $\Ann_R M$.
\end{theorem}
\begin{proof}
Let $M\in\T$ and choose a finite presentation
\[
 R^a\xrightarrow{d}R^b\longrightarrow M\longrightarrow0.
\]
Let $R_0\subseteq R$ be the $\mathbb Z$-subalgebra generated by the entries of $d$. It is a finitely generated $\mathbb Z$-algebra and hence is noetherian. Taking the cokernel of the same matrix over $R_0$ gives a finitely presented $R_0$-module $M_0$ with $R\otimes_{R_0}M_0\cong M$. Consider the class
\[
 \mathcal D=\{L\in\modcat R_0:R\otimes_{R_0}L\in\T\}.
\]
Tensoring a finite presentation shows that $R\otimes_{R_0}L$ is finitely presented over $R$ whenever $L\in\modcat R_0$. Thus the defining condition for $\mathcal D$ takes place in $\modcat R$, and $M_0\in\mathcal D$.

The functor $R\otimes_{R_0}-$ is right exact between the abelian categories $\modcat R_0$ and $\modcat R$. Thus $\mathcal D$ is a torsion class by \cite[Lemma~3.2(a)]{IK}.

Since $R_0$ is commutative and noetherian, the theorem of Stanley and Wang \cite{SW}, also proved in \cite[Corollary~3.23(b)]{IK}, shows that $\mathcal D$ is Serre. Choose generators $m_1,\ldots,m_s$ of the $R_0$-module $M_0$. Then the map
\[
 R_0\longrightarrow M_0^s,\quad
 a\longmapsto(am_1,\ldots,am_s)
\]
 induces an injection
$R_0/\Ann_{R_0}M_0\hookrightarrow M_0^s$.
Since $M_0^s\in\mathcal D$ and $\mathcal D$ is closed under subobjects, we obtain $R_0/\Ann_{R_0}M_0\in\mathcal D$. Tensoring this module gives
\(
 R/(\Ann_{R_0}M_0)R\in\T.
\)
Every element of $\Ann_{R_0}M_0$ annihilates $R\otimes_{R_0}M_0\cong M$, so
$(\Ann_{R_0}M_0)R\subseteq\Ann_R M$. Consequently there is a surjection
\[
 R/(\Ann_{R_0}M_0)R\longrightarrow R/\Ann_R M\longrightarrow0.
\]
 Closure under quotients now gives $R/\Ann_R M\in\T$.

Finally, let $N\in\modcat R$ satisfy $(\Ann_R M)N=0$. A finite generating set of $N$ gives a surjection $(R/\Ann_R M)^t\twoheadrightarrow N$, which implies that $N\in\T$. In particular, every subobject $L\subseteq M$ in $\modcat R$ is annihilated by $\Ann_R M$ and hence belongs to $\T$. Thus $\T$ is closed under subobjects as well as quotients and extensions, proving that it is Serre.
\end{proof}

\section{Perfect realization and support classification}\label{sec:perfect}

For every perfect complex $P$ and a prime ideal $\mathfrak p$, exactness of localization gives $H_i(P_{\mathfrak p})\cong H_i(P)_{\mathfrak p}$. Thus $P_{\mathfrak p}\simeq0$ if and only if $H_i(P)_{\mathfrak p}=0$ for every $i$. Consequently,
\begin{equation}\label{eq:homology-support}
 \Supp_R P=\bigcup_i\Supp_R H_i(P).
\end{equation}

Write $\K_R(\boldsymbol a)$ for the homological Koszul complex on $\boldsymbol a=(a_1,\ldots,a_r)$: it is the total tensor product $\bigotimes_{j=1}^r[R\xrightarrow{a_j}R]$, with each factor in degrees $1,0$ and the usual tensor differential \cite[Section~4.5, pp.~111--112]{Weibel}. Thus it lies in degrees $0,\ldots,r$; use $\K_R(\varnothing)=R[0]$. Multiplication by every $a_j$ on this complex is null-homotopic \cite[(2.9)(a) and its proof]{DGI}, so the ideal $(\boldsymbol a)=(a_1,\ldots,a_r)$ annihilates all its homology.

\begin{theorem}\label{thm:witness}
Let $M\in\modcat R$, let $\boldsymbol a=(a_1,\ldots,a_r)$ generate $\Ann_R M$, and choose a finite free presentation $R^a\xrightarrow{d}R^b\twoheadrightarrow M$. Then the bounded finite free complex
\begin{equation}\label{eq:witness}
 X_M=[R^a\xrightarrow{d}R^b]\otimes_R\K_R(\boldsymbol a),
\end{equation}
with presentation complex in degrees $1,0$, satisfies
\begin{equation}\label{eq:witness-properties}
 H_0(X_M)\cong M,\quad H_j(X_M)\in\ab_R(M)\quad(j\in\Z),\quad
 \Supp_R X_M=\Supp_R M.
\end{equation}
For $M=0$, one may instead take $X_M=0$.
\end{theorem}
\begin{proof}
Put $C=[R^a\xrightarrow{d}R^b]$, with $C_1=R^a$ and $C_0=R^b$. We first compute the zeroth homology of $X_M=C\otimes_R\K_R(\boldsymbol a)$. The complex has the form
\[
 X_M:\quad\cdots\longrightarrow (X_M)_2\longrightarrow
 \underbrace{R^a\oplus(R^b)^r}_{\text{degree }1}
 \xrightarrow{\partial_1}
 \underbrace{R^b}_{\text{degree }0}\longrightarrow0,
\]
where $\partial_1(u,v_1,\ldots,v_r)=d(u)+\sum_{j=1}^r a_jv_j$. Thus $\im\partial_1=\im d+(\Ann_R M)R^b$, and
\[
 H_0(X_M)=R^b/\bigl(\im d+(\Ann_R M)R^b\bigr)
 \cong M/\bigl((\Ann_R M)M\bigr)=M.
\]
We next show that $\Ann_R M$ annihilates every homology module of $X_M$. For a generator $a_j$, let $s_j$ be a Koszul homotopy satisfying $d_{\K}s_j+s_jd_{\K}=a_j\operatorname{id}_{\K}$. On a homogeneous tensor with $c\in C_p$, define
\(
 h_j(c\otimes k)=(-1)^p c\otimes s_j(k).
\)
Using the tensor differential, the two terms involving $d_C(c)$ cancel, and we obtain
\[
 (d_{X_M}h_j+h_jd_{X_M})(c\otimes k)
 =c\otimes(d_{\K}s_j+s_jd_{\K})(k)
 =a_j(c\otimes k).
\]
Thus multiplication by $a_j$ induces zero on homology. Applying this to each generator gives $(\Ann_R M)H_i(X_M)=0$ for every $i$. Since $X_M$ is a bounded complex of finite free modules and $R$ is coherent, its cycles, boundaries and homology modules are finitely presented. Equation~\eqref{eq:ab-generator} therefore gives $H_i(X_M)\in\ab_R(M)$.

It follows that $\Supp_R H_i(X_M)\subseteq V(\Ann_R M)$ for every $i$, and hence $\Supp_R X_M\subseteq V(\Ann_R M)$ by~\eqref{eq:homology-support}. Conversely, $H_0(X_M)\cong M$ gives $\Supp_R M\subseteq\Supp_R X_M$. Since $\Supp_R M=V(\Ann_R M)$, these inclusions prove the support equality in~\eqref{eq:witness-properties}.
\end{proof}

This construction supplies the perfect complex sought in the discussion after \cite[Corollary~3.2]{Hovey}, with the stronger requirement that all homology belongs to $\ab_R(M)$. It imposes no finite projective dimension assumption on $M$. Takahashi first realizes cyclic modules by Koszul complexes and then uses a finite cyclic filtration to reconstruct a general module \cite[Lemmas~3.4--3.5 and proof of Theorem~3.6]{Takahashi}. Theorem~\ref{thm:witness} instead realizes any finitely presented $M$ as $H_0$ of one explicit perfect complex. The closure assumptions differ: Takahashi's cyclic lemma requires only finite sums and cokernels, whereas the abelian generation used here also allows kernels.

For a Thomason subset $V$, we write
\[
 \W_V=\{M\in\modcat R:\Supp_R M\subseteq V\}\quad \text{and} \quad
 \T_V=\{P\in\Dperf(R):\Supp_R P\subseteq V\}.
\]
The noetherian case of the perfect-complex classification is due to Hopkins and Neeman \cite[Theorem~1.5]{Neeman}. We use Thomason's generalization to arbitrary commutative rings \cite[Theorem~3.15]{Thomason}: the assignments
\begin{equation}\label{eq:thomason}
 \T\longmapsto\bigcup_{P\in\T}\Supp_R P,
 \quad V\longmapsto\T_V
\end{equation}
are inverse order-preserving bijections between thick subcategories and Thomason subsets. Every thick subcategory of $\Dperf(R)$ is a \emph{tensor ideal}: it contains $P\otimes_R^{\mathbf L}Q$ whenever it contains $P$ and $Q\in\Dperf(R)$. This is the affine case of \cite[Corollary~3.11.1(a)]{Thomason}.

\begin{proof}[\bf Proof of Theorem~\ref{thm:main}]
\textup{(1).} The equalities were proved in Section~\ref{sec:finite}, using Proposition~\ref{prop:finite-submodules} and Theorem~\ref{thm:torsion-serre}. The remaining isomorphism is the correspondence of Garkusha and Prest \cite[Corollary~2.4 and its proof]{GP}: a Serre subcategory is sent to its filtered-colimit closure, and the inverse is $\mathcal U\mapsto\mathcal U\cap\modcat R$.

\textup{(2).} Since $R$ is coherent, Hovey's results show that $f_R$ is well defined and surjective \cite[Proposition~1.2, Lemma~1.6 and Corollary~2.2]{Hovey}. We use Theorem~\ref{thm:witness} to prove injectivity and identify the inverse in~\eqref{eq:inverse}.

For every $P\in f_R(\W)$, the definition gives $H_0(P)\in\W$. Conversely, if $M\in\W$, then $\ab_R(M)\subseteq\W$. Theorem~\ref{thm:witness} therefore gives $X_M\in f_R(\W)$ with $H_0(X_M)\cong M$. These two inclusions prove
\begin{equation}\label{eq:homology-recovery}
 \{H_0(P):P\in f_R(\W)\}=\W,
\end{equation}
and hence injectivity of $f_R$.

Now let $\T$ be any thick subcategory. By surjectivity, choose a wide subcategory $\W$ with $f_R(\W)=\T$. Equation~\eqref{eq:homology-recovery} gives $g_R(\T)=\W$, and hence $f_R(g_R(\T))=\T$. Thus $g_R$ is the inverse of $f_R$. Both maps preserve inclusion, so they are complete lattice isomorphisms. The values of $g_R$ are Serre by~(1). Moreover, for $M\in g_R(\T)=\W$, Theorem~\ref{thm:witness} gives $X_M\in f_R(\W)=\T$ with $H_0(X_M)\cong M$. Thus bounded finite free representatives suffice in~\eqref{eq:inverse}.

\textup{(3).} By~(1) and \cite[Theorem~B]{GP}, the maps $\W\mapsto\Supp_R(\W)$ and $V\mapsto\W_V$ are inverse order-preserving bijections. Thus the left diagonal in~\eqref{eq:main-lattices} is a complete lattice isomorphism. The right diagonal is the isomorphism~\eqref{eq:thomason}, and the horizontal arrows are inverse complete lattice isomorphisms by~(2).

It remains to check commutativity. For every Thomason subset $V$, equation~\eqref{eq:homology-support} gives
\[
 P\in f_R(\W_V)
 \quad\Longleftrightarrow\quad
 \Supp_R H_i(P)\subseteq V\text{ for every }i
 \quad\Longleftrightarrow\quad P\in\T_V.
\]
Hence $f_R(\W_V)=\T_V$. Since $\Supp_R(\W_V)=V=\Supp_R(\T_V)$ and every wide subcategory is of the form $\W_V$, the diagram commutes.
\end{proof}

\begin{remark}\label{rem:witness-generators}
The complex $X_M$ in~\eqref{eq:witness} also records the size of the construction and the abelian subcategory generated by its homology. Its terms are
\[
 (X_M)_j\cong R^{\,b\binom rj+a\binom r{j-1}}\quad(0\le j\le r+1),
\]
where binomial coefficients outside their usual range are zero. Indeed, only $R^b\otimes_R\K_R(\boldsymbol a)_j$ and $R^a\otimes_R\K_R(\boldsymbol a)_{j-1}$ contribute to total degree $j$, and $\K_R(\boldsymbol a)_j\cong R^{\binom rj}$.

Moreover,
\begin{equation}\label{eq:ab-homology}
 \ab_R\{H_j(X_M):j\in\Z\}
 =\ab_R\{H_j(\K_R(\boldsymbol a)):j\in\Z\}=\ab_R(M).
\end{equation}
Indeed, all the homology modules of $X_M$ belong to $\ab_R(M)$, while $H_0(X_M)\cong M$ gives the reverse inclusion. For the Koszul complex, its homology modules are finitely presented and annihilated by $\Ann_R M$, so they too belong to $\ab_R(M)$. Its zeroth homology is $R/\Ann_R M$, which generates $\ab_R(M)$ by~\eqref{eq:ab-generator}. This proves both equalities.
\end{remark}

We record consequences of the support classification for generation and finite lattice operations. Recall that an element $c$ of a complete lattice is \emph{compact} if $c\le\bigvee_\lambda c_\lambda$ implies $c\le\bigvee_{\lambda\in F}c_\lambda$ for some finite $F$.

\begin{corollary}\label{cor:compact}
For any family $\mathcal S\subseteq\modcat R$,
\[
 \wide_R(\mathcal S)=\left\{N\in\modcat R:\Supp_R N\subseteq\bigcup_{M\in\mathcal S}\Supp_R M\right\}.
\]
The compact elements of $\Wide(\modcat R)$ are precisely the subcategories $\wide_R(M)$, equivalently the subcategories $\wide_R(R/I)$ for finitely generated ideals $I$. They generate the lattice by joins. For finitely generated ideals $I,J$, one has
\begin{equation}\label{eq:finite-lattice}
 \begin{aligned}
  \wide_R(R/I)\cap\wide_R(R/J)&=\wide_R\bigl(R/(I+J)\bigr),\\
  \wide_R(R/I)\vee\wide_R(R/J)&=\wide_R(R/IJ).
 \end{aligned}
\end{equation}
For $\Ann_R M=(\boldsymbol a)$ one also has
\[
 f_R(\wide_R(M))=\thick_R(\K_R(\boldsymbol a))=\thick_R(X_M).
\]
\end{corollary}
\begin{proof}
The support formula follows from Theorem~\ref{thm:main}(3), since joins of Thomason subsets are unions. Equation~\eqref{eq:wide-family} shows that membership in a generated wide subcategory depends on a finite subfamily of generators. Hence each $\wide_R(M)$ is compact. Conversely,
\[
 \W=\bigvee_{M\in\W}\wide_R(M),
\]
so a compact $\W$ is a finite join of such subcategories and is generated by the direct sum of the corresponding modules. By~\eqref{eq:ab-generator}, $\wide_R(M)=\wide_R(R/\Ann_R M)$, which gives the equivalent description by the subcategories $\wide_R(R/I)$ with $I$ finitely generated.

Under the support isomorphism, $\wide_R(R/I)$ corresponds to $V(I)$. The identities $V(I)\cap V(J)=V(I+J)$ and $V(I)\cup V(J)=V(IJ)$ therefore give~\eqref{eq:finite-lattice}.

Finally, the Koszul complex on $\boldsymbol a$ and the complex $X_M$ both have support $V(\Ann_R M)$, by the Koszul homology annihilation property and Theorem~\ref{thm:witness}. Thomason's classification~\eqref{eq:thomason} identifies their thick closures with $f_R(\wide_R(M))$.
\end{proof}


\section{Change of rings}\label{sec:basechange}

The support classification in Theorem~\ref{thm:main} controls change of rings without a flatness assumption. We first describe the induced lattice maps, then compare the lengths of quotient-and-extension filtrations.


Let $\varphi:R\to S$ be a homomorphism of commutative coherent rings, and let $\pi:\Spec S\to\Spec R$ be the induced map. For $\W\in\Wide(\modcat R)$ and $\T\in\Thick(\Dperf(R))$, we define
\[
 \varphi_!\W=\wide_S\{S\otimes_R M:M\in\W\}\quad \text{and}\quad
 \varphi_!\T=\thick_S\{S\otimes_R^{\mathbf L}P:P\in\T\}.
\]
\begin{proposition}\label{prop:basechange}
The following diagram commutes, and each horizontal arrow is an isomorphism of complete lattices:
\begin{equation}\label{eq:basechange}
 \begin{tikzcd}[column sep=normal,row sep=normal]
 \Wide(\modcat R)
   \arrow[r,"f_R","\sim"']
   \arrow[d,"\varphi_!"']
 & \Thick(\Dperf(R))
   \arrow[r,"\Supp_R","\sim"']
   \arrow[d,"\varphi_!"]
 & \Thom(\Spec R)\arrow[d,"\pi^{-1}"]
 \\
 \Wide(\modcat S)\arrow[r,"f_S","\sim"']
 & \Thick(\Dperf(S))\arrow[r,"\Supp_S","\sim"']
 & \Thom(\Spec S)
 \end{tikzcd}
\end{equation}
Both operations $\varphi_!$ preserve arbitrary joins, finite meets and compact elements; for a finitely generated ideal $I$,
\[
 \varphi_!\wide_R(R/I)=\wide_S(S/IS).
\]
Each map $\varphi_!$ is an order embedding if and only if $\pi$ is surjective. Under this condition, for every $M\in\modcat R$,
\begin{equation}\label{eq:detection}
 M\in\W\quad\Longleftrightarrow\quad S\otimes_R M\in\varphi_!\W.
\end{equation}
In particular,~\eqref{eq:detection} holds for faithfully flat homomorphisms. If $J$ is a nil ideal and $R/J$ is coherent, the quotient map induces isomorphisms on both lattices. This applies to every finitely generated nilpotent ideal.
\end{proposition}
\begin{proof}
We first compute the support of a base-changed module. Let $\mathfrak q\in\Spec S$ and put $\mathfrak p=\varphi^{-1}(\mathfrak q)$. If $M$ is finitely generated over $R$, then $(S\otimes_R M)_{\mathfrak q}$ is finitely generated over the local ring $S_{\mathfrak q}$. Nakayama's lemma gives
\[
 (S\otimes_R M)_{\mathfrak q}\ne0
 \quad\Longleftrightarrow\quad M\otimes_R\kappa(\mathfrak q)\ne0.
\]
On the other hand,
\[
 M\otimes_R\kappa(\mathfrak q)
 \cong (M_{\mathfrak p}/\mathfrak pM_{\mathfrak p})
       \otimes_{\kappa(\mathfrak p)}\kappa(\mathfrak q).
\]
Extension of a field detects whether a vector space is zero. Applying Nakayama's lemma over $R_{\mathfrak p}$ therefore shows that the last module is nonzero if and only if $M_{\mathfrak p}\ne0$. Consequently
\[
 \Supp_S(S\otimes_R M)=\pi^{-1}\Supp_R M.
\]

For perfect complexes, Thomason's base-change formula \cite[Lemma~3.3(b)]{Thomason} gives
\[
 \Supp_S(S\otimes_R^{\mathbf L}P)=\pi^{-1}\Supp_R P.
\]

By Theorem~\ref{thm:main} and Corollary~\ref{cor:compact}, the support of a generated wide subcategory is the union of the supports of its generators. The same assertion for a generated thick subcategory follows from~\eqref{eq:thomason}. Thus the two support identities give
\[
 \Supp_S(\varphi_!\W)=\pi^{-1}\Supp_R(\W)\quad\text{and}\quad
 \Supp_S(\varphi_!\T)=\pi^{-1}\Supp_R(\T).
\]
Theorem~\ref{thm:main} also gives $\Supp_R(f_R(\W))=\Supp_R(\W)$, and similarly over $S$. Therefore $f_S(\varphi_!\W)$ and $\varphi_!(f_R(\W))$ have the same support and are equal by~\eqref{eq:thomason}. This proves commutativity of~\eqref{eq:basechange}; its horizontal arrows are complete lattice isomorphisms by Theorem~\ref{thm:main}.

Inverse image preserves arbitrary unions and finite intersections, which are the corresponding joins and meets of Thomason subsets. Moreover, if $I$ is finitely generated, then $IS$ is finitely generated and $\pi^{-1}V(I)=V(IS)$. Corollary~\ref{cor:compact} therefore proves preservation of compact elements, as well as
\[
 \varphi_!\wide_R(R/I)=\wide_S(S/IS).
\]

Suppose that $\pi$ is surjective. If $U,V\subseteq\Spec R$ satisfy $\pi^{-1}U\subseteq\pi^{-1}V$, choose, for each $\mathfrak p\in U$, a prime $\mathfrak q$ with $\pi(\mathfrak q)=\mathfrak p$. Then $\mathfrak q\in\pi^{-1}V$, so $\mathfrak p\in V$. Hence inverse image reflects inclusion, and both maps $\varphi_!$ are order embeddings. Writing $U=\Supp_R(\W)$, the support classifications give
\[
 \begin{aligned}
 S\otimes_R M\in\varphi_!\W
 &\quad\Longleftrightarrow\quad
 \pi^{-1}\Supp_R M\subseteq\pi^{-1}U\\
 &\quad\Longleftrightarrow\quad \Supp_R M\subseteq U\\
 &\quad\Longleftrightarrow\quad M\in\W.
 \end{aligned}
\]
This proves~\eqref{eq:detection}. A faithfully flat homomorphism induces a surjection on spectra, so it satisfies this condition.

Conversely, suppose that $\mathfrak p\in\Spec R$ is not in the image of $\pi$. The primes of the fiber ring $S\otimes_R\kappa(\mathfrak p)$ correspond to primes of $S$ contracting to $\mathfrak p$. Thus the fiber ring has empty spectrum and is zero. Since this ring is $(R\setminus\mathfrak p)^{-1}(S/\mathfrak pS)$, there is some $u\in R\setminus\mathfrak p$ with $\varphi(u)\in\mathfrak pS$. Write
\[
 \varphi(u)=\sum_{j=1}^r\varphi(a_j)s_j
 \quad\text{with }a_j\in\mathfrak p,\ s_j\in S,
\]
and put $I=(a_1,\ldots,a_r)$. Then $I\subseteq\mathfrak p$ and $\varphi(u)\in IS$, so $V(IS)\subseteq V(uS)$. However, $\mathfrak p\in V(I)$ and $\mathfrak p\notin V(u)$, giving
\[
 \pi^{-1}V(I)\subseteq\pi^{-1}V(u)
 \quad \text{and} \quad V(I)\nsubseteq V(u).
\]
Both $V(I)$ and $V(u)$ are Thomason subsets. Thus inverse image fails to reflect inclusion on the Thomason lattice, and commutativity of~\eqref{eq:basechange} shows that neither map $\varphi_!$ is an order embedding.

Finally, since $J$ is nil, every prime ideal of $R$ contains $J$. The map $\Spec(R/J)\to\Spec R$ is therefore a homeomorphism. A homeomorphism preserves closed subsets with quasi-compact complement and their unions, so inverse image is an isomorphism of the Thomason lattices. If $R/J$ is coherent, the horizontal isomorphisms in~\eqref{eq:basechange} give the asserted isomorphisms for wide and thick subcategories. In particular, when $J$ is finitely generated and nilpotent, it is a nil ideal and $R/J$ is coherent, by \cite[Example~4.61(c)]{Lam}.
\end{proof}


To compare filtration lengths under change of rings, we first describe quotient-and-extension generation over a fixed commutative coherent ring $R$.

Write $\gen_R(\mathcal S)$ for the quotient objects of finite direct sums of objects of $\mathcal S$, with quotients taken in $\modcat R$. 

We record the elementary filtration operations used below; compare the quotient-and-extension construction in \cite[Lemma~3.1]{MS}. Filtrations in any abelian category are understood in the same sense as above.

\Needspace{8\baselineskip}
\begin{lemma}\label{lem:filtration-operations}
Let $\mathcal A$ be a full subcategory of an abelian category $\mathcal E$, containing $0$ and closed under finite direct sums and quotients. For $m,n\ge1$ the following hold.
\begin{enumerate}[label=\textup{(\arabic*)},leftmargin=*]
\item $\Filt_n(\mathcal A)$ is closed under finite direct sums and quotients.
\item An extension of an object of $\Filt_n(\mathcal A)$ by an object of $\Filt_m(\mathcal A)$ belongs to $\Filt_{m+n}(\mathcal A)$, and
\[
 \Filt_m(\Filt_n(\mathcal A))\subseteq\Filt_{mn}(\mathcal A).
\]
Consequently $\Filt(\mathcal A)$ is the smallest subcategory containing $\mathcal A$ and closed under quotients and extensions.
\item Let $F:\mathcal E\to\mathcal E'$ be an additive right-exact functor between abelian categories, and let $\mathcal B$ be a full subcategory of $\mathcal E'$ that is closed under quotients and contains every $F(A)$ with $A\in\mathcal A$. If $N\in\Filt_n(\mathcal A)$, then $F(N)\in\Filt_n(\mathcal B)$.
\end{enumerate}
\end{lemma}
\begin{proof}
\textup{(1).} Let $q:N\twoheadrightarrow Q$ be an epimorphism and choose a filtration $0=N_0\subseteq\cdots\subseteq N_n=N$ with factors in $\mathcal A$, inserting zero factors if necessary. Its image filtration has terms $q(N_i)$. At the first step, $q(N_1)$ is a quotient of $N_1$ and belongs to $\mathcal A$. If $n\ge2$, the surjection
\[
 N_2/N_1\twoheadrightarrow q(N_2)/q(N_1)
\]
shows that the next factor belongs to $\mathcal A$. Repeating this step gives the same conclusion for every factor, proving quotient closure. For finite direct sums, extend the given filtrations to length $n$ by inserting zero factors if necessary and take their direct sum term by term. Its successive factors are finite direct sums of objects of $\mathcal A$.

\textup{(2).} In an exact sequence $0\to L\to E\xrightarrow{p}Q\to0$, extend a filtration of $L$ by taking the inverse images under $p$ of a filtration of $Q$. The resulting factors are the original factors of $L$ and $Q$, so their lengths add. To prove the inclusion, start with a filtration $0=N_0\subseteq\cdots\subseteq N_m=N$ whose factors lie in $\Filt_n(\mathcal A)$. The first term $N_1$ has a filtration of length at most $n$. If $m\ge2$, lift a filtration of $N_2/N_1$ to $N_2$ and use it to extend the filtration of $N_1$, obtaining length at most $2n$. Repeating this step gives length at most $mn$. Thus $\Filt(\mathcal A)$ is closed under quotients by (1) and under extensions by the first assertion. Any extension-closed subcategory containing $\mathcal A$ contains every finite filtration by its objects, proving minimality.

\textup{(3).} Apply $F$ to a filtration $0=N_0\subseteq\cdots\subseteq N_n=N$ and put $L_i=\im(F(N_i)\to F(N))$. Then $L_0=0$ and $L_n=F(N)$. Right exactness supplies a surjection
\[
 F(N_i/N_{i-1})\twoheadrightarrow L_i/L_{i-1}
\]
at every step. Its source belongs to $\mathcal B$, so every factor belongs to $\mathcal B$ by closure under quotients. The image filtration therefore proves the assertion.
\end{proof}

For $M\in\modcat R$, define
\[
 b_R(M)=\min\{m\ge1:R/\Ann_R M\in\Filt_m(\gen_R(M))\},
\]
with $b_R(M)=\infty$ if the set is empty. The next corollary shows that it is finite.

\begin{corollary}\label{cor:quotient-generation}
For any family $\mathcal S\subseteq\modcat R$,
\begin{equation}\label{eq:quotient-generation}
 \begin{aligned}
 \tors_R(\mathcal S)&=\wide_R(\mathcal S)=\Filt(\gen_R(\mathcal S))\\
 &=\{N\in\modcat R:I_F^nN=0\text{ for some finite }F\subseteq\mathcal S,\ n\ge1\},
 \end{aligned}
\end{equation}
where $I_F$ is as in Proposition~\ref{prop:generation}. For $M\in\modcat R$, the integer $b_R(M)$ is finite, and for every $n\ge1$ one has
\begin{equation}\label{eq:filtration-comparison}
 \Filt_n(\gen_R(M))\ \subseteq\ \{N\in\modcat R:(\Ann_R M)^nN=0\}
 \ \subseteq\ \Filt_{b_R(M)n}(\gen_R(M)).
\end{equation}
The integer $b_R(M)$ is the least positive integer coefficient for which the right-hand inclusion holds for every $n\ge1$.
\end{corollary}
\begin{proof}
Lemma~\ref{lem:filtration-operations}(2) gives $\tors_R(\mathcal S)=\Filt(\gen_R(\mathcal S))$. Theorem~\ref{thm:main}(1) identifies this torsion closure with $\wide_R(\mathcal S)$, and~\eqref{eq:wide-family} gives the annihilator description. In particular, $R/\Ann_R M\in\wide_R(M)=\Filt(\gen_R(M))$, so $b_R(M)$ is finite.

For the first inclusion in~\eqref{eq:filtration-comparison}, Proposition~\ref{prop:generation} gives $\gen_R(M)\subseteq\ab_R(M)$ and
\[
 \Filt_n(\gen_R(M))\subseteq\Filt_n(\ab_R(M))
 =\{N\in\modcat R:(\Ann_R M)^nN=0\}.
\]

For the reverse inclusion, put $b=b_R(M)$. Lemma~\ref{lem:filtration-operations}(1) shows that $\Filt_b(\gen_R(M))$ is closed under finite direct sums and quotients.

By definition, $R/\Ann_R M\in\Filt_b(\gen_R(M))$. If $L\in\ab_R(M)$, then $(\Ann_R M)L=0$ by~\eqref{eq:ab-generator}, so a finite generating set gives a surjection $(R/\Ann_R M)^u\twoheadrightarrow L$. The closure properties just proved yield
\[
 \ab_R(M)\subseteq\Filt_b(\gen_R(M)).
\]
If $(\Ann_R M)^nN=0$, equation~\eqref{eq:filtration} and Lemma~\ref{lem:filtration-operations}(2) now give
\[
 N\in\Filt_n(\ab_R(M))
 \subseteq\Filt_n\bigl(\Filt_b(\gen_R(M))\bigr)
 \subseteq\Filt_{bn}(\gen_R(M)).
\]
This proves the second inclusion in~\eqref{eq:filtration-comparison}.

Finally, suppose a positive integer $c$ can replace $b_R(M)$ in that inclusion for every $n\ge1$. Taking $n=1$ and $N=R/\Ann_R M$ gives $R/\Ann_R M\in\Filt_c(\gen_R(M))$. The definition of $b_R(M)$ then gives $b_R(M)\le c$, which proves the asserted minimality.
\end{proof}

The filtration length admits a description by the standard iteration of module traces; see, for example, \cite[Example~2.2 and Section~3.2.1]{Conde}. We check that its terms remain finitely presented in the present setting.
For $M,N\in\modcat R$, write
\[
 \operatorname{tr}_M(N)=\sum_{f\in\Hom_R(M,N)}f(M).
\]
Starting with $T_0^M(N)=0$, define $T_{i+1}^M(N)$ as the inverse image of
$\operatorname{tr}_M(N/T_i^M(N))$ under the quotient map from $N$.
When the ring needs to be specified, write $T_{i,R}^M(N)$.

\begin{proposition}\label{prop:trace-filtration}
For $M,N\in\modcat R$, every $T_i^M(N)$ is finitely presented, and for $n\ge1$,
\[
 N\in\Filt_n(\gen_R(M))\quad\Longleftrightarrow\quad T_n^M(N)=N.
\]
Consequently, 
\[
 b_R(M)=\min\{n\ge1:T_n^M(R/\Ann_R M)=R/\Ann_R M\}.
\]
\end{proposition}
\begin{proof}
We first show that the trace is finitely presented. Applying $\Hom_R(-,N)$ to a finite presentation $R^a\to R^b\to M\to0$ gives an exact sequence
\[
 0\longrightarrow\Hom_R(M,N)\longrightarrow N^b\longrightarrow N^a.
\]
Thus $\Hom_R(M,N)$ is finitely presented by coherence. Choose generators $f_1,\ldots,f_r$ of this $R$-module. Every map $f:M\to N$ is an $R$-linear combination of the $f_j$, so its image is contained in $\sum_j f_j(M)$. It follows that
\[
 \operatorname{tr}_M(N)=\sum_{j=1}^r f_j(M)
=\im\bigl(M^r\longrightarrow N,\ (m_1,\ldots,m_r)\longmapsto
 \textstyle\sum_j f_j(m_j)\bigr).
\]
This image is finitely generated, and hence finitely presented by coherence; it also belongs to $\gen_R(M)$.

In particular, $T_1^M(N)=\operatorname{tr}_M(N)$ is finitely presented and belongs to $\gen_R(M)$. Suppose $T_i^M(N)$ is finitely presented. Then $N/T_i^M(N)$ is finitely presented, and the preceding argument applies to its trace. By the definition of $T_{i+1}^M(N)$, there is an exact sequence
\[
 0\longrightarrow T_i^M(N)\longrightarrow T_{i+1}^M(N)
 \longrightarrow\operatorname{tr}_M(N/T_i^M(N))\longrightarrow0.
\]
Both end terms are finitely presented, so the middle term is finitely presented as well, and its quotient by $T_i^M(N)$ lies in $\gen_R(M)$. Repeating this step proves the assertions about all terms and successive factors. Consequently, if $T_n^M(N)=N$, then
\[
 0=T_0^M(N)\subseteq T_1^M(N)\subseteq\cdots\subseteq T_n^M(N)=N
\]
is a filtration showing that $N\in\Filt_n(\gen_R(M))$.

Conversely, let $N\in\Filt_n(\gen_R(M))$ and choose a filtration
$0=N_0\subseteq N_1\subseteq\cdots\subseteq N_n=N$, inserting zero factors if necessary. We show successively that $N_i\subseteq T_i^M(N)$. At the first step, $N_1\in\gen_R(M)$, so there is a surjection $M^s\twoheadrightarrow N_1$. Composing with $N_1\hookrightarrow N$ expresses $N_1$ as a sum of images of maps $M\to N$. Thus
\[
 N_1\subseteq\operatorname{tr}_M(N)=T_1^M(N).
\]
Now suppose $N_i\subseteq T_i^M(N)$, and let $B$ be the image of $N_{i+1}$ in $N/T_i^M(N)$. The map $N_{i+1}\to B$ kills $N_i$ and therefore induces a surjection $N_{i+1}/N_i\twoheadrightarrow B$. Since $N_{i+1}/N_i\in\gen_R(M)$, we also have $B\in\gen_R(M)$. Applying the first-step argument to the submodule $B\subseteq N/T_i^M(N)$ gives
\(
 B\subseteq\operatorname{tr}_M(N/T_i^M(N)).
\)
Taking inverse images in $N$ yields $N_{i+1}\subseteq T_{i+1}^M(N)$. Repeating the step until $i+1=n$ gives $N=N_n\subseteq T_n^M(N)\subseteq N$, and hence $T_n^M(N)=N$.

Applying this equivalence to $N=R/\Ann_R M$ gives the formula for $b_R(M)$. The minimum is finite by Corollary~\ref{cor:quotient-generation}.
\end{proof}

\begin{example}\label{ex:strict-filtration}
The factor $b_R(M)$ cannot always be replaced by $1$. To see this, let
\[
 R=k[x,y,z_1,z_2,\ldots],\quad M=(x,y),
\]
where $k$ is a field. The ring $R$ is coherent by \cite[Example~4.61(a)]{Lam}, and the strictly increasing chain $(z_1)\subsetneq(z_1,z_2)\subsetneq\cdots$ shows that it is not noetherian.

The presentation
\[
 R\xrightarrow{(-y,x)}R^2\xrightarrow{(x,y)}M\longrightarrow0
\]
shows that $M$ is finitely presented, and $\Ann_R M=0$ because $R$ is a domain and $M\ne0$. For any map $f:M\to R$, the relation $yx=xy$ in $M$ gives $yf(x)=xf(y)$. Reducing this equality modulo $x$ and using that $R/(x)$ is a domain gives $f(x)=xc$ for some $c\in R$. Substitution and cancellation of $x$ then give $f(y)=yc$. Hence every map $M\to R$ has image contained in $(x,y)$. The same holds for a map from any finite sum of copies of $M$, so no such map surjects onto $R$. Therefore $R\notin\gen_R(M)$, and $b_R(M)>1$.

On the other hand, the assignments $x\mapsto1+(x,y)$ and $y\mapsto0$ respect the relation $(-y,x)$ and define a surjection $M\twoheadrightarrow R/(x,y)$. Both factors in the filtration
\[
 0\subseteq M\subseteq R
\]
therefore belong to $\gen_R(M)$: the first is $M$, and the second is $R/M\cong R/(x,y)$. This gives $R\in\Filt_2(\gen_R(M))$, so $b_R(M)=2$. The example distinguishes the exact annihilator formula~\eqref{eq:filtration} from the comparison~\eqref{eq:filtration-comparison}.

\end{example}



We now compare these filtration lengths under change of rings.

\begin{corollary}\label{cor:filtration-basechange}
Let $\varphi:R\to S$ be a homomorphism of commutative coherent rings.
For $M,N\in\modcat R$ and $n\ge1$,
\begin{equation}\label{eq:filtration-basechange}
 N\in\Filt_n(\gen_R(M))
 \quad\Longrightarrow\quad
 S\otimes_R N\in\Filt_n(\gen_S(S\otimes_R M)).
\end{equation}
Moreover,
\[
 b_S(S\otimes_R M)\le b_R(M).
\]
If $\varphi$ is faithfully flat, the implication in
\eqref{eq:filtration-basechange} is an equivalence and this inequality is
an equality.
\end{corollary}
\begin{proof}
Put $M_S=S\otimes_R M$ and $N_S=S\otimes_R N$. The functor $S\otimes_R-$ is additive and right exact from $\modcat R$ to $\modcat S$. It sends $\gen_R(M)$ into $\gen_S(M_S)$, since it preserves finite direct sums and epimorphisms. Applying Lemma~\ref{lem:filtration-operations}(3) gives~\eqref{eq:filtration-basechange}.

To prove the inequality, put $b=b_R(M)$. The defining filtration of $R/\Ann_R M$ has at most $b$ factors in $\gen_R(M)$. Applying the implication just proved gives
\[
 S/\bigl((\Ann_R M)S\bigr)\cong S\otimes_R(R/\Ann_R M)
 \in\Filt_b(\gen_S(M_S)).
\]
Every element of $(\Ann_R M)S$ annihilates $M_S$, so $(\Ann_R M)S\subseteq\Ann_S M_S$ and there is a surjection
\[
 S/\bigl((\Ann_R M)S\bigr)\twoheadrightarrow S/\Ann_S M_S.
\]
Lemma~\ref{lem:filtration-operations}(1) therefore gives $S/\Ann_S M_S\in\Filt_b(\gen_S(M_S))$, and hence $b_S(M_S)\le b_R(M)$.

Now assume that $S$ is flat over $R$. We use the standard compatibility of traces with flat base change; compare \cite[Lemma~1.5(iii)]{HHS} for trace ideals. For a finitely presented $R$-module $L$, choose a finite presentation $R^a\to R^b\to M\to0$. Applying $\Hom_R(-,L)$ gives
\[
 0\longrightarrow\Hom_R(M,L)\longrightarrow L^b\longrightarrow L^a.
\]
Tensoring this sequence with $S$ preserves its kernel by flatness. Comparing it with the sequence obtained by applying $\Hom_S(-,S\otimes_R L)$ to the base-changed presentation of $M$ proves that the natural map
\[
 S\otimes_R\Hom_R(M,L)
 \longrightarrow \Hom_S(M_S,S\otimes_R L)
\]
is an isomorphism. The trace $\operatorname{tr}_M(L)$ is the image of the evaluation map $\Hom_R(M,L)\otimes_R M\to L$. Under the preceding isomorphism, its base change is the evaluation map for $M_S$ and $S\otimes_R L$. Since flat tensoring preserves images, we obtain
\[
 S\otimes_R\operatorname{tr}_M(L)
 =\operatorname{tr}_{M_S}(S\otimes_R L)
\]
as submodules of $S\otimes_R L$.

We next compare the iterated traces. Both filtrations start with zero, and the trace identity for $L=N$ gives equality at the first step. Suppose that equality has been proved at step $i$. Flatness then gives an identification
\[
 S\otimes_R(N/T_{i,R}^M(N))
 \cong N_S/T_{i,S}^{M_S}(N_S).
\]
Apply the trace identity to $L=N/T_{i,R}^M(N)$. By definition, $T_{i+1,R}^M(N)$ is the inverse image of $\operatorname{tr}_M(L)$ under $N\to L$, or equivalently the kernel of $N\to L/\operatorname{tr}_M(L)$. Tensoring this kernel description and using flatness identifies its image in $N_S$ with $T_{i+1,S}^{M_S}(N_S)$. Repeating this step proves
\[
 S\otimes_R T_{i,R}^M(N)=T_{i,S}^{M_S}(N_S)
 \quad(i\ge0).
\]

If $S$ is faithfully flat, an $R$-module is zero if and only if its tensor product with $S$ is zero. Hence
\[
 \begin{aligned}
 T_{n,S}^{M_S}(N_S)=N_S
 &\quad\Longleftrightarrow\quad
 S\otimes_R(N/T_{n,R}^M(N))=0\\
 &\quad\Longleftrightarrow\quad T_{n,R}^M(N)=N.
 \end{aligned}
\]
Proposition~\ref{prop:trace-filtration} now gives the reverse implication in~\eqref{eq:filtration-basechange}.

Finally, choose generators $m_1,\ldots,m_s$ of $M$. Tensoring~\eqref{eq:ann-kernel} with the flat module $S$ identifies $(\Ann_R M)S$ with the kernel of
\[
 S\longrightarrow M_S^s,\quad
 a\longmapsto(a(1\otimes m_1),\ldots,a(1\otimes m_s)).
\]
Since $1\otimes m_1,\ldots,1\otimes m_s$ generate $M_S$, this kernel is $\Ann_S M_S$. Thus $\Ann_S M_S=(\Ann_R M)S$. Applying the filtration equivalence to $N=R/\Ann_R M$, whose base change is $S/\Ann_S M_S$, shows that the same positive integers occur in the two definitions of $b_R(M)$ and $b_S(M_S)$. Therefore,
$b_R(M)=b_S(M_S)$.  
\end{proof}

The inequality can be strict even when the spectra are homeomorphic. Let $R=k[x,y]/(x,y)^2$ over a field $k$, put $\mathfrak m=(x,y)$, and take $M=R/(x)\oplus R/(y)$. As $k$-vector spaces, $(x)=kx$ and $(y)=ky$, so their intersection is zero. Hence
\[
 \Ann_R M=(x)\cap(y)=0.
\]
A map $R/(x)\to R$ is determined by the image of $1+(x)$, which can be any element of $\Ann_R(x)=\mathfrak m$. Its possible images therefore sum to $\mathfrak m$. The same argument applies to $R/(y)$, and thus
\[
 T_1^M(R)=\operatorname{tr}_M(R)=\mathfrak m\ne R.
\]
The natural surjection $R/(x)\twoheadrightarrow R/\mathfrak m$ shows that $\operatorname{tr}_M(R/\mathfrak m)=R/\mathfrak m$. At the second step, the inverse image of this trace under $R\to R/\mathfrak m$ is all of $R$, so $T_2^M(R)=R$. Proposition~\ref{prop:trace-filtration} gives $b_R(M)=2$. In contrast, for the nilpotent quotient $S=R/\mathfrak m=k$, we have $S\otimes_R M\cong S^2$. Projection onto either summand gives $S\in\gen_S(S^2)$, and hence $b_S(S\otimes_R M)=1$. Proposition~\ref{prop:basechange} gives lattice isomorphisms for this quotient, but the filtration length has decreased.

\section{Arbitrary modules and the limits of finite reconstruction}\label{sec:large}

We now examine how much a wide subcategory of arbitrary modules is determined by its finitely presented objects. The starting point is the Garkusha--Prest correspondence recalled in Theorem~\ref{thm:main}(1). We first record its element and support descriptions, then compare its image with all wide subcategories closed under arbitrary sums.


For $\W\in\Wide(\modcat R)$, write $\overrightarrow{\W}$ for its closure under filtered colimits in $\Mod R$. Let $\mathcal F$ be the multiplicative filter corresponding to $\W$ in Theorem~\ref{thm:filters}, and put
\[
 V_{\mathcal F}=\bigcup_{I\in\mathcal F}V(I),\quad
 \widehat{\mathcal F}
 =\{J\lhd R:I\subseteq J\text{ for some }I\in\mathcal F\}.
\]
Here $\mathcal F$ consists of finitely generated ideals, whereas $\widehat{\mathcal F}$ is its upward closure among all ideals. By \cite[Proposition~2.1]{GPtorsion}, $\widehat{\mathcal F}$ is a Gabriel filter of finite type. For a hereditary torsion class, its \emph{torsion radical} assigns to each module its largest submodule in that class; see \cite[Chapter~VI]{Stenstrom} and \cite[Appendix]{GPtorsion}.

\begin{lemma}[{\cite[Corollary~2.4]{GP} and \cite[Theorem~2.2]{GPtorsion}}]\label{lem:torsion}
With the notation above,
\begin{align}
 \overrightarrow{\W}
 &=\{X\in\Mod R:\text{for every }x\in X\text{ there is }I\in\mathcal F\text{ with }Ix=0\}\label{eq:torsion-elements}\\
 &=\{X\in\Mod R:\Supp_R X\subseteq V_{\mathcal F}\}.\label{eq:torsion-support}
\end{align}
This is the smallest hereditary torsion class containing $\W$, and its torsion radical is
\[
 t_{\mathcal F}(X)=\{x\in X:Ix=0\text{ for some }I\in\mathcal F\}.
\]
\end{lemma}
\begin{proof}
The finitely generated ideals in the Gabriel filter attached to $\overrightarrow{\W}$ are precisely those $I$ with $R/I\in\W$, hence are exactly $\mathcal F$. Thus that Gabriel filter is $\widehat{\mathcal F}$, and its element and support descriptions give the displayed formulas. Minimality follows because every hereditary torsion class is closed under filtered colimits.
\end{proof}

Over a noetherian ring every hereditary torsion class is of finite type \cite[p.~768]{GP}, so this correspondence recovers the torsion-class part of \cite[Theorem~A]{Takahashi}. Example~\ref{ex:boundaries} below illustrates why finite type is needed in the coherent case.

Over left noetherian rings, Angeleri H\"ugel--Sentieri show that directed-colimit closure preserves wide subcategories of finitely presented modules and recovers the original subcategory on restriction \cite[Theorem~4.6]{AHS}. We describe the image of this construction and its relation to arbitrary wide subcategories closed under sums in the present commutative coherent setting.

Let $\Wide_{\oplus}(\Mod R)$ denote the lattice of wide subcategories closed under arbitrary direct sums. For $\W\in\Wide(\modcat R)$ and $\C\in\Wide_{\oplus}(\Mod R)$, set
\[
 \lambda_R(\W)=\overrightarrow{\W},\quad
 \rho_R(\C)=\C\cap\modcat R.
\]

\begin{lemma}\label{lem:large-adjunction}
The order-preserving maps $\lambda_R$ and $\rho_R$ fit into the diagram
\[
 \begin{tikzcd}[column sep=huge,row sep=large]
 \Wide(\modcat R)
   \arrow[r,shift left=.7ex,"\lambda_R"]
   \arrow[dr,"\W\mapsto\overrightarrow{\W}"',"\sim"]
 & \Wide_{\oplus}(\Mod R)
   \arrow[l,shift left=.7ex,"\rho_R"]
 \\
 & \Torsft(\Mod R)\arrow[u,hook]
 \end{tikzcd}
\]
where $\lambda_R$ is the composite of the diagonal lattice isomorphism and the vertical inclusion. The horizontal maps form an adjoint pair:
\begin{equation}\label{eq:large-adjunction}
 \lambda_R(\W)\subseteq\C
 \quad\Longleftrightarrow\quad
 \W\subseteq\rho_R(\C),\quad \rho_R\lambda_R=\mathrm{id}.
\end{equation}
The subcategory
\[
 \lambda_R\rho_R(\C)=\overrightarrow{\C\cap\modcat R}\subseteq\C
\]
is the largest finite-type hereditary torsion class contained in $\C$. The equality $\lambda_R\rho_R(\C)=\C$ holds precisely when $\C$ is itself a finite-type hereditary torsion class, equivalently when it is generated by its finitely presented objects using kernels, cokernels, extensions and arbitrary direct sums.
\end{lemma}
\begin{proof}
First, since $R$ is coherent, we have that $\rho_R(\C)$ is wide, and Theorem~\ref{thm:main}(1) gives $\lambda_R(\W)\in\Torsft(\Mod R)$ and $\rho_R\lambda_R=\mathrm{id}$. Moreover,  a subcategory closed under arbitrary sums and cokernels is closed under colimits. Thus $\W\subseteq\C$ implies $\overrightarrow{\W}\subseteq\C$; the reverse implication follows from $\W\subseteq\overrightarrow{\W}$. This proves the adjunction.

In particular, $\lambda_R\rho_R(\C)\subseteq\C$. For any $\mathcal U\in\Torsft(\Mod R)$ contained in $\C$, Theorem~\ref{thm:main}(1) gives
\[
 \mathcal U=\lambda_R(\mathcal U\cap\modcat R)
 \subseteq\lambda_R\rho_R(\C).
\]
This proves maximality and the first equality criterion. The adjunction also identifies $\lambda_R(\W)$ with the smallest wide subcategory closed under arbitrary sums and containing $\W$. Taking $\W=\rho_R(\C)$ gives the stated generation criterion.
\end{proof}


To compare this adjunction with \cite{Krause}, assume that $R$ is noetherian. For an $R$-module $M$, regarded as a complex concentrated in degree zero, define its \emph{small support} by
\[
 \supp_R M=\{\mathfrak p\in\Spec R:M\otimes_R^{\mathbf L}\kappa(\mathfrak p)\not\simeq0\}.
\]
For finitely generated $M$, one has $\supp_R M=\Supp_R M$ by \cite[Section~3, p.~737]{Krause}; these supports can differ for arbitrary modules. Following \cite[Section~3]{Krause}, a subset $\Phi\subseteq\Spec R$ is \emph{coherent} if every morphism $I^0\to I^1$ of injective modules with $\Ass_R I^i\subseteq\Phi$ extends to an exact sequence $I^0\to I^1\to I^2$ with $I^2$ injective and $\Ass_R I^2\subseteq\Phi$. Here $\Ass_R$ is the ordinary associated-prime set defined in the introduction. Krause's classification \cite[Theorem~3.1]{Krause} identifies such subsets with
\[
 \C_{\Phi}=\{X\in\Mod R:\supp_R X\subseteq\Phi\}
 \in\Wide_{\oplus}(\Mod R).
\]
The following consequence of that classification and Lemma~\ref{lem:large-adjunction} describes exactly what restriction to finitely presented modules retains.

For a subset $\Phi\subseteq\Spec R$, put
\[
 \Phi^{\circ}_{\mathrm{sp}}
 =\{\mathfrak p\in\Spec R:V(\mathfrak p)\subseteq\Phi\},
\]
the largest specialization-closed subset contained in $\Phi$.

\begin{proposition}\label{prop:krause-trace}
Let $R$ be commutative and noetherian, and let $\Phi$ be coherent in Krause's sense. Then
\begin{align}
 \rho_R(\C_{\Phi})
 &=\{M\in\modcat R:\Supp_R M\subseteq\Phi^{\circ}_{\mathrm{sp}}\},\label{eq:finite-trace}\\
 \lambda_R\rho_R(\C_{\Phi})
 &=\{X\in\Mod R:\Supp_R X\subseteq\Phi^{\circ}_{\mathrm{sp}}\}.\label{eq:torsion-part}
\end{align}
In particular, $\C_{\Phi}$ is generated by its finitely presented objects under kernels, cokernels, extensions and arbitrary sums if and only if $\Phi$ is specialization-closed. The map $\rho_R$ on all of $\Wide_{\oplus}(\Mod R)$ is injective if and only if $R$ is Artinian.
\end{proposition}
\begin{proof}
For $M\in\modcat R$, small and ordinary supports coincide by \cite[Section~3, p.~737]{Krause}. Moreover, $\Supp_R M=V(\Ann_R M)$ is closed, and hence specialization-closed. Thus $\Supp_R M\subseteq\Phi$ is equivalent to $\Supp_R M\subseteq\Phi^{\circ}_{\mathrm{sp}}$, proving~\eqref{eq:finite-trace}. The modules $R/\mathfrak p$ for $\mathfrak p\in\Phi^{\circ}_{\mathrm{sp}}$ show that the union of these supports is $\Phi^{\circ}_{\mathrm{sp}}$. The support description in Lemma~\ref{lem:torsion} then gives~\eqref{eq:torsion-part}.

Krause's classification identifies hereditary torsion classes with specialization-closed subsets \cite[Corollary~3.6]{Krause}. Over a noetherian ring these torsion classes are of finite type, so Lemma~\ref{lem:large-adjunction} gives the generation criterion.

If $R$ is Artinian, every subset of $\Spec R$ is specialization-closed. Hence $\lambda_R\rho_R(\C)=\C$ for every $\C\in\Wide_{\oplus}(\Mod R)$, and $\rho_R$ is injective. Conversely, a non-Artinian noetherian ring has a nonmaximal prime $\mathfrak p$. The singleton $\{\mathfrak p\}$ is coherent \cite[Corollary~4.4]{Krause}, and $\C_{\{\mathfrak p\}}\ne0$ because it contains $\kappa(\mathfrak p)$. However, $\{\mathfrak p\}^{\circ}_{\mathrm{sp}}=\varnothing$, so~\eqref{eq:finite-trace} gives $\rho_R(\C_{\{\mathfrak p\}})=0=\rho_R(0)$. Thus $\rho_R$ is not injective.
\end{proof}

For example, Krause's classification applied to the coherent singleton $\{(0)\}\subseteq\Spec\mathbb Z$ gives $\C_{\{(0)\}}=\Mod\mathbb Q$. Its finitely presented part is zero by~\eqref{eq:finite-trace}, although it contains $\mathbb Q$. It does not contain the submodule $\mathbb Z\subseteq\mathbb Q$, so it is not Serre. Here
\[
 \supp_{\mathbb Z}\mathbb Q=\{(0)\},\quad
 \Supp_{\mathbb Z}\mathbb Q=\Spec\mathbb Z.
\]
This illustrates both the distinction between the two supports and the loss of information under $\rho_R$.

\begin{example}\label{ex:boundaries}
Let $A=\mathbb C[x_1,x_2,\ldots]$ and $\mathfrak m=(x_1,x_2,\ldots)$. The ring $A$ is coherent by \cite[Example~4.61(a)]{Lam}, whereas $\{\mathfrak m\}$ is the standard non-Thomason subset in \cite[Section~2]{GPtorsion}. The nonzero hereditary torsion class
\[
 \mathcal U=\{X\in\Mod A:\Supp_A X\subseteq\{\mathfrak m\}\}
\]
is therefore not of finite type, by \cite[Theorem~2.2]{GPtorsion}. Moreover, $\mathcal U\cap\modcat A=0$: the support of a nonzero finitely presented module is a nonempty Thomason subset, so it cannot be contained in $\{\mathfrak m\}$. Thus the finite-type hypothesis is necessary for reconstruction even among hereditary torsion classes.
\end{example}

\appendix
\section{Associated primes and finite presentation}\label{app:associated-primes}

For a commutative noetherian ring, Takahashi's classification identifies subsets of the spectrum with subcategories of finitely generated modules closed under subobjects and extensions \cite[Theorem~4.1]{Takahashi}. The following observation describes the obstruction to this parametrization over a coherent ring.

Let $R$ be a commutative coherent ring, let $\mathfrak S(R)$ be the lattice of full subcategories of $\modcat R$ closed under subobjects and extensions, and let $\mathcal P(E)$ denote the power set of a set $E$. Set
\[
 \begin{aligned}
 \Spec_{\mathrm{fg}}R&=\{\mathfrak p\in\Spec R:\mathfrak p\text{ is finitely generated}\},\\
 \Phi_R(S)&=\{M\in\modcat R:\Ass_R M\subseteq S\}
 \end{aligned}
\]
for $S\subseteq\Spec R$.

\Needspace{9\baselineskip}
\begin{proposition}\label{prop:ass-boundary}
The map $\Phi_R:\mathcal P(\Spec R)\to\mathfrak S(R)$ is well-defined, and
\begin{equation}\label{eq:visible-primes}
 \bigcup_{M\in\modcat R}\Ass_R M=\Spec_{\mathrm{fg}}R.
\end{equation}
For all $S,T\subseteq\Spec R$,
\begin{equation}\label{eq:ass-order}
 \Phi_R(S)\subseteq\Phi_R(T)
 \quad\Longleftrightarrow\quad
 S\cap\Spec_{\mathrm{fg}}R\subseteq T\cap\Spec_{\mathrm{fg}}R.
\end{equation}
Consequently, $\Phi_R$ is injective if and only if $R$ is noetherian, equivalently if and only if $\Phi_R$ is a complete lattice isomorphism. In the noetherian case its inverse is $\C\mapsto\bigcup_{M\in\C}\Ass_R M$.
\end{proposition}
\begin{proof}
The standard inclusions for associated primes under submodules and short exact sequences show that $\Phi_R(S)$ is closed under subobjects and extensions \cite[Proposition~2.3(2),(3)]{ES}.

To prove~\eqref{eq:visible-primes}, let $\mathfrak p=\Ann_R(x)\in\Ass_R M$ with $M\in\modcat R$. The kernel of the map $R\to M$, $r\mapsto rx$, is finitely generated by coherence, so $\mathfrak p\in\Spec_{\mathrm{fg}}R$. Conversely, if $\mathfrak p\in\Spec_{\mathrm{fg}}R$, then $R/\mathfrak p\in\modcat R$ and $\Ass_R(R/\mathfrak p)=\{\mathfrak p\}$ \cite[Proposition~2.3(1)]{ES}.

If $\Phi_R(S)\subseteq\Phi_R(T)$, testing on $R/\mathfrak p$ for $\mathfrak p\in S\cap\Spec_{\mathrm{fg}}R$ gives $\mathfrak p\in T$. Conversely, if $S\cap\Spec_{\mathrm{fg}}R\subseteq T\cap\Spec_{\mathrm{fg}}R$, then~\eqref{eq:visible-primes} gives
\[
 \Ass_R M\subseteq S\cap\Spec_{\mathrm{fg}}R\subseteq T
 \quad(M\in\Phi_R(S)),
\]
so $\Phi_R(S)\subseteq\Phi_R(T)$. This proves~\eqref{eq:ass-order}.

It follows that $\Phi_R$ is injective exactly when every prime ideal is finitely generated: a prime outside $\Spec_{\mathrm{fg}}R$ gives $\Phi_R(\{\mathfrak p\})=\Phi_R(\varnothing)$. Cohen's criterion identifies this condition with noetherianness \cite[p.~3017]{LR}. In that case Takahashi's classification cited above gives the complete lattice isomorphism and its inverse.
\end{proof}

For the non-finitely generated maximal ideal $\mathfrak m$ of Example~\ref{ex:boundaries}, Proposition~\ref{prop:ass-boundary} gives $\Phi_A(\{\mathfrak m\})=\Phi_A(\varnothing)$.

\section*{Acknowledgments}
This work was supported by NSFC (No. 12471036) and Hubei
Provincial Natural Science Foundation of China (No. 2026AFA094).


\begin{thebibliography}{99}

\bibitem{AHS}
L.~Angeleri H\"ugel and F.~Sentieri,
\emph{Wide coreflective subcategories and torsion pairs},
J. Algebra \textbf{664} (2025), 164--205.
\href{https://doi.org/10.1016/j.jalgebra.2024.10.025}{doi:10.1016/j.jalgebra.2024.10.025}.

\bibitem{AP}
S.~Asai and C.~Pfeifer,
\emph{Wide subcategories and lattices of torsion classes},
Algebr. Represent. Theory \textbf{25} (2022), no.~6, 1611--1629.
\href{https://doi.org/10.1007/s10468-021-10079-1}{doi:10.1007/s10468-021-10079-1}.

\bibitem{Conde}
T.~Conde,
\emph{$\Delta$-filtrations and projective resolutions for the Auslander--Dlab--Ringel algebra},
Algebr. Represent. Theory \textbf{21} (2018), no.~3, 605--625.
\href{https://doi.org/10.1007/s10468-017-9730-z}{doi:10.1007/s10468-017-9730-z}.

\bibitem{DGI}
W.~Dwyer, J.~P.~C.~Greenlees and S.~Iyengar,
\emph{Finiteness in derived categories of local rings},
Comment. Math. Helv. \textbf{81} (2006), no.~2, 383--432.
\href{https://doi.org/10.4171/CMH/56}{doi:10.4171/CMH/56}.

\bibitem{ES}
N.~Epstein and J.~Shapiro,
\emph{Strong Krull primes and flat modules},
J. Pure Appl. Algebra \textbf{218} (2014), no.~9, 1712--1729.
\href{https://doi.org/10.1016/j.jpaa.2014.01.009}{doi:10.1016/j.jpaa.2014.01.009}.

\bibitem{GP}
G.~Garkusha and M.~Prest,
\emph{Classifying Serre subcategories of finitely presented modules},
Proc. Amer. Math. Soc. \textbf{136} (2008), no.~3, 761--770.
\href{https://doi.org/10.1090/S0002-9939-07-08844-2}{doi:10.1090/S0002-9939-07-08844-2}.

\bibitem{GPtorsion}
G.~Garkusha and M.~Prest,
\emph{Torsion classes of finite type and spectra},
in \emph{K-Theory and Noncommutative Geometry},
EMS Ser. Congr. Rep., European Math. Soc., Z\"urich, 2008, pp.~393--412.
\href{https://doi.org/10.4171/060-1/12}{doi:10.4171/060-1/12}.

\bibitem{Glaz}
S.~Glaz,
\emph{Commutative Coherent Rings},
Lecture Notes in Math., vol.~1371, Springer-Verlag, Berlin, 1989.
\href{https://doi.org/10.1007/BFb0084570}{doi:10.1007/BFb0084570}.

\bibitem{HHS}
J.~Herzog, T.~Hibi and D.~I.~Stamate,
\emph{The trace of the canonical module},
Israel J. Math. \textbf{233} (2019), 133--165.
\href{https://doi.org/10.1007/s11856-019-1898-y}{doi:10.1007/s11856-019-1898-y}.

\bibitem{Hopkins}
M.~J.~Hopkins,
\emph{Global methods in homotopy theory},
in \emph{Homotopy Theory (Durham, 1985)},
London Math. Soc. Lecture Note Ser., vol.~117, Cambridge University Press, Cambridge, 1987, pp.~73--96.
\href{https://doi.org/10.1017/CBO9781107325746.005}{doi:10.1017/CBO9781107325746.005}.

\bibitem{Hovey}
M.~Hovey,
\emph{Classifying subcategories of modules},
Trans. Amer. Math. Soc. \textbf{353} (2001), no.~8, 3181--3191.
\href{https://doi.org/10.1090/S0002-9947-01-02747-7}{doi:10.1090/S0002-9947-01-02747-7}.

\bibitem{IMST}
K.~Iima, H.~Matsui, K.~Shimada and R.~Takahashi,
\emph{When is a subcategory Serre or torsion-free?},
Publ. Res. Inst. Math. Sci. \textbf{60} (2024), no.~4, 831--857.
\href{https://doi.org/10.4171/PRIMS/60-4-7}{doi:10.4171/PRIMS/60-4-7}.

\bibitem{IK}
O.~Iyama and Y.~Kimura,
\emph{Classifying subcategories of modules over Noetherian algebras},
Adv. Math. \textbf{446} (2024), Article 109631, 62 pp.
\href{https://doi.org/10.1016/j.aim.2024.109631}{doi:10.1016/j.aim.2024.109631}.

\bibitem{Kanda}
R.~Kanda,
\emph{Classification of categorical subspaces of locally noetherian schemes},
Doc. Math. \textbf{20} (2015), 1403--1465.
\href{https://doi.org/10.4171/DM/522}{doi:10.4171/DM/522}.

\bibitem{Krause}
H.~Krause,
\emph{Thick subcategories of modules over commutative noetherian rings}
(with an appendix by S.~Iyengar),
Math. Ann. \textbf{340} (2008), no.~4, 733--747.
\href{https://doi.org/10.1007/s00208-007-0166-3}{doi:10.1007/s00208-007-0166-3}.

\bibitem{Lam}
T.~Y.~Lam,
\emph{Lectures on Modules and Rings},
Graduate Texts in Math., vol.~189, Springer-Verlag, New York, 1999.
\href{https://doi.org/10.1007/978-1-4612-0525-8}{doi:10.1007/978-1-4612-0525-8}.

\bibitem{LR}
T.~Y.~Lam and M.~L.~Reyes,
\emph{A Prime Ideal Principle in commutative algebra},
J. Algebra \textbf{319} (2008), 3006--3027.
\href{https://doi.org/10.1016/j.jalgebra.2007.07.016}{doi:10.1016/j.jalgebra.2007.07.016}.

\bibitem{MS}
F.~Marks and J.~\v{S}\v{t}ov\'{i}\v{c}ek,
\emph{Torsion classes, wide subcategories and localisations},
Bull. Lond. Math. Soc. \textbf{49} (2017), no.~3, 405--416.
\href{https://doi.org/10.1112/blms.12033}{doi:10.1112/blms.12033}.

\bibitem{Neeman}
A.~Neeman,
\emph{The chromatic tower for $D(R)$}
(with an appendix by M.~B\"okstedt),
Topology \textbf{31} (1992), no.~3, 519--532.
\href{https://doi.org/10.1016/0040-9383(92)90047-L}{doi:10.1016/0040-9383(92)90047-L}.

\bibitem{SW}
D.~Stanley and B.~Wang,
\emph{Classifying subcategories of finitely generated modules over a Noetherian ring},
J. Pure Appl. Algebra \textbf{215} (2011), no.~11, 2684--2693.
\href{https://doi.org/10.1016/j.jpaa.2011.03.013}{doi:10.1016/j.jpaa.2011.03.013}.

\bibitem{Stenstrom}
B.~Stenstr\"om,
\emph{Rings of Quotients: An Introduction to Methods of Ring Theory},
Grundlehren Math. Wiss., vol.~217, Springer-Verlag, Berlin, 1975.
\href{https://doi.org/10.1007/978-3-642-66066-5}{doi:10.1007/978-3-642-66066-5}.

\bibitem{Takahashi}
R.~Takahashi,
\emph{Classifying subcategories of modules over a commutative noetherian ring},
J. Lond. Math. Soc. (2) \textbf{78} (2008), no.~3, 767--782.
\href{https://doi.org/10.1112/jlms/jdn056}{doi:10.1112/jlms/jdn056}.

\bibitem{TWZ}
L.~Tan, D.~Wang and T.~Zhao,
\emph{Wide subcategories in abelian categories},
Front. Math. \textbf{21} (2026), no.~1, 203--214.
\href{https://doi.org/10.1007/s11464-021-0489-5}{doi:10.1007/s11464-021-0489-5}.

\bibitem{Thomason}
R.~W.~Thomason,
\emph{The classification of triangulated subcategories},
Compositio Math. \textbf{105} (1997), no.~1, 1--27.
\href{https://doi.org/10.1023/A:1017932514274}{doi:10.1023/A:1017932514274}.

\bibitem{Weibel}
C.~A.~Weibel,
\emph{An Introduction to Homological Algebra},
Cambridge Studies in Adv. Math., vol.~38, Cambridge University Press, Cambridge, 1994.
\href{https://doi.org/10.1017/CBO9781139644136}{doi:10.1017/CBO9781139644136}.

\end{thebibliography}
\end{document}